\documentclass[12pt]{amsart}

\usepackage{url}
\usepackage{verbatim}
\usepackage[margin=1in]{geometry} 
\usepackage{amsmath,amscd, amsthm, mathrsfs, amssymb, esint}
\usepackage{graphicx}
\usepackage{float}
\usepackage{subcaption}
\usepackage{physics}
\usepackage{mathtools}
\mathtoolsset{showonlyrefs}

\usepackage[colorlinks=true,linkcolor=red]{hyperref}

\makeatletter
\renewcommand*{\ref}[1]{%
  \hyperref[{#1}]{\textup{\tagform@{\ref*{#1}}}}%
}
\makeatother
\usepackage{array}
\usepackage{booktabs}
\theoremstyle{plain}
   \newtheorem{theorem}{Theorem}[section]

\theoremstyle{definition}

   \newtheorem{proposition}{Proposition}[section]

   \newtheorem{lemma}{Lemma}[section]
   \newtheorem{corollary}{Corollary}[section]
   \newtheorem{conjecture}{Conjecture}[section]

   \newtheorem{definition}{Definition}[section]
   
   \newtheorem{remark}{Remark}[section]

\newcommand{\vv}{\mathbf{v}}

\newcommand{\RR}{\mathbb{R}}

\newcommand{\mini}[2]{\text{min}\{#1,#2\}}

\newcommand{\N}{\mathbb{N}} 
\newcommand{\R}{\mathbb{R}} 
\newcommand{\E}{\mathcal{E}}

\renewcommand\ip[2]{\langle#1, #2\rangle}

\newcommand{\inner}[2]{\left \langle #1, #2\right \rangle}
\newcommand{\lap}{\Delta}

\newcommand{\sg}{SG}

\def\scriptO{{{\it O}\kern -.42em {\it `}\kern + .20em}}
\usepackage{esvect}

\begin{document}

\title{Sobolev orthogonal polynomials on the Sierpinski gasket}

\author{Qingxuan Jiang} 
\address{Department of Mathematics, 
Malott Hall, Cornell University, Ithaca, NY 14853, USA}
\email{qj46@cornell.edu}
\author{Tian Lan}
\address{Department of Mathematics,
ETH Zurich, Ramistrasse 101, 8092 Zurich, Switzerland}
\email{tialan@student.ethz.ch}
\author{Kasso A.~Okoudjou}
\address{Department of Mathematics,
Tufts University, Medford MA 02131, USA} \email{Kasso.Okoudjou@tufts.edu}
\author{Robert S. Strichartz}
\address{Department of Mathematics, 
Malott Hall, Cornell University, Ithaca, NY 14853, USA} \email{str@math.cornell.edu}
\author{Shashank Sule}
\address{Department of Mathematics,
University of Maryland, College Park, MD 20742, USA} \email{ssule25@umd.edu}
\author{Sreeram Venkat}
\address{Department of Mathematics, SAS HAll, North Carolina State University, 2311 Stinson Dr, Raleigh, NC 27607, USA} \email{srvenkat@ncsu.edu}
\author{Xiaoduo Wang}
\address{Department of Mathematics, University of Illinois Urbana-Champaign, Urbana, IL 61801, USA} \email{xiaduo3@illinois.edu}


\date{\today}

\maketitle

\begin{abstract}
We develop a theory of Sobolev orthogonal polynomials on the Sierpi\'nski gasket ($\sg$), which is a fractal set that can be viewed as a limit of a sequence of finite graphs.  These orthogonal polynomials arise through the Gram-Schmidt orthogonalisation process applied on the set of monomials on $\sg$ using  several notions of a Sobolev inner products. After establishing some recurrence relations for these orthogonal polynomials, we give  estimates for their  $L^2$, $L^\infty$ and Sobolev norms, and study  their asymptotic behaviour.  Finally, we study the properties of zero sets of polynomials and develop fast computational tools to explore applications to quadrature and interpolation.

\end{abstract}

\maketitle \pagestyle{myheadings} \thispagestyle{plain}
\markboth{Q. JIANG, T. LIAN, K. A. OKOUDJOU, R. S. STRICHARTZ, S. SULE, S. VENKAT, AND X. WANG}{SOBOLEV ORTHOGONAL POLYNOMIALS ON SG}

\tableofcontents

\section{Introduction}\label{sec: intro}
Over the last two decades, a theory of calculus on fractal sets such as the Sierpi\'nski gasket ($\sg$) has been developed and is based on the analysis of the fractal Laplacian \cite{NSTY, BGSKSY, SU}. In particular, a polynomial of degree $j$ on  $\sg$ is any solution of the equation $\Delta^{j+1}u=0$, where $\Delta$ denotes the Laplacian on $\sg$ \cite{kigami2001analysis, S}. While most aspects of the theory of polynomials in this setting parallel their counterparts on the unit interval $I=[0, 1]$, there exists a number of striking differences. In particular, there is no analog of the Weierstrass Theorem on $\sg$, that is the set of polynomials on $\sg$ is not complete on $L^2(\sg)$ \cite[Theorem 4.3.6]{kigami2001analysis}. On the other hand, an initial theory of orthogonal polynomials has been developed on $\sg$ and resulted in an analog of the Legendre orthogonal polynomials on $[-1,1]$, \cite{OST}. More specifically,  the  Legendre OPs on $\sg$ arise in solving the least squares problem $$\arg\min\{\|f-g\|_{2}\,\, \,  g\,\,  \text{polynomial\, of\, order}\, j\},$$where   $f\in L^2(\sg).$ 

In this paper, we are interested in solving a similar least squares problem, with the added requirement that the function to approximate is also smooth. That is we seek the solution of the optimization problem 
$$\arg\min\{\|f-g\|_{S}\,\, \,  g\,\,  \text{polynomial\, of\, order\, at\, most}\, j\},$$ where $f\in S$ a linear subspace of $L^2(\sg)$ that measures the smoothness of $f$. We will show that the solution to this problem can be expressed in terms of orthogonal polynomials with respect to various inner products.  We shall generically refer to these OPs  as \emph{Sobolev Orthogonal Polynomials (SOP)}. Our aim is to initiate a systematic study of SOPs on $\sg$ in analogy to the theory of Sobolev OPs  on $\R$. 

We begin by a brief review of the SOP on $[-1,1]$, and we refer to \cite{mmb17, MX, Meij} for more details on this class of OPs and certain of their generalizations.
Given $\chi>0$, consider the inner product  $$\inner{f}{g}_{S} := \int_{-1}^{1}fg\,dx + \chi\int_{-1}^{1}f'g'\,dx$$
defined on the Sobolev space $W^{1,2}([-1,1])$ of functions $f\in L^2[-1,1]$ such that $f'\in L^2[-1,1]$. Applying the Gram-Schimdt orthogonalisation process to the monomials $\{x^n\}_{n\geq 0}$ in this inner product space results in the so-called  Sobolev-Legendre polynomials, which we  denote by  $\qty{S_{n}\qty(\cdot; \chi)}$. We refer to  \cite{Alt, sch72} where these OPs were first investigated.
    While the classical Legendre OPs $\qty{P_n}$ on $[-1,1]$ satisfy the ubiquitous three-term recurrence formula, there exists no such relation for the  Sobolev-Legendre OPs. This can be seen as the consequence of the fact that $$\ip{xf}{g}_S\neq \ip{f}{xg}_S.$$  However, the Sobolev-Legendre OPs enjoy the following properties. 

\begin{enumerate}
    \item The two-term differential equation
    $$ \chi S''\qty(x;\chi) - S\qty(x;\chi) = A_nP'_{n+1}\qty(x) + B_nP'_{n-1}\qty(x)$$ for some constants $A_n, B_n.$
    \item The two-term recurrence relation 
    $$ S_{n}\qty(x;\chi) - S_{n-2}\qty(x;\chi) = a_{n}\qty(P_{n}\qty(x)) - P_{n-2}\qty(x))$$ for some sequence $a_n$.
    \item $\qty{S_n\qty(x;\chi)}$ has $n$ simple zeroes in $\qty(-1,1)$
\end{enumerate}

Our goal is to use the theory of polynomials on $SG$ that was developed in \cite{S} to investigate OPs on $SG$ with respect to the family of inner products given by $$\ip{f}{g}_{W^{m, 2}}=\int_{SG}f(x)g(x)\, \dd{\mu}(x)+\sum_{k=1}^m \chi_k\int_{SG}\lap^k{f(x)}\lap^k{g(x)}\, \dd{\mu} (x),$$ where $\chi_k$ are nonnegative integers and $m\geq 1$. The case  $m=1$ will be our model case which we shall discuss in details and compare to the situation on the interval $[-1, 1]$.

We recall that the polynomials on $SG$ can be build from three basic families of monomials that we will defined formally in Section~\ref{sec:polynomials}. Consequently, we construct  the (three families)  OPs with respect to the inner products given above by  applying the Gram-Schmidt orthogonalisation process. Subsequently, we prove several properties of the resulting OPs. Some of these properties are common to the three families, while others are different.  For example, all three families of OPs satisfy a three-term recurrence relation and a three-term differential equation involving the Legendre Orthogonal Polynomials on $\sg$ studied in \cite{OST}. The recurrence relations also allow us to establish explicit bounds on various norms of the SOPs, and to derive other interesting results.

Furthermore, by combining  the three-term recurrence and the computational procedures developed in \cite{OST}, we graph several SOPs not only on $\sg$, but also on its edges. The visualization of these polynomials allows a detailed study of their qualitative properties, such as the number and location of zeroes. Observing that some polynomials seem to have more zeroes than the dimension of their ambient subspace, we raise a more fundamental question on finding sets of points  which allow interpolation of functions on $\sg$. We are able to establish that these sets must have empty interior, and proceed to construct an infinite family of such sets. We use these constructions to give a generalization of the spline quadrature formula originally developed in \cite{SU}.

The rest of the paper is organized as follows. In Section \ref{sec:prelim}, we introduce the analytical tools required to define  the polynomials on the $\sg$ and we prove a general topological result pertaining to the location of zeroes of continuous functions on $\sg$. The section also introduces polynomial spaces on $\sg$ and a formal definition of the  Sobolev inner products whose properties are presented.
This section also establishes results on the zeroes of entire functions defined in \cite{NSTY}. These results, while not directly related to the SOPs are of interest in their own right. Section \ref{sec:sopsg} contains most of the main results of this paper, dealing with the three families of SOPs. In addition, we prove a recurrence relation and a number of related results  for a generalized Sobolev inner product involving higher-order derivatives and boundary terms. Finally, in Section \ref{sec: applications} we present the aforementioned plots of the polynomials and discuss applications to interpolation and quadrature on $\sg$.

We conclude this introduction by pointing out that along with \cite{OST}, this paper can be viewed as not only laying the foundation of a general theory of OPs on $\sg$, but also initiate some applications of such a theory. Furthermore, because $\sg$ can be viewed as a limit  of a sequence of finite graphs, our results also suggest developing this theory for  graphs, and explore the implications for graph signal processing. For more on polynomials approximation of graph signals see, e.g.,  \cite{Pesenson09, WNW20}.

\section{Polynomials on  \texorpdfstring{$\sg$}{SG}}\label{sec:prelim} 
In the first part of this section, we collect the analytical tools on $\sg$ that are needed to prove our results. In particular, the definition of $\sg$ as a limit of a sequence of finite graphs will play an important role in some of our results. More details can be found in \cite{S}. The second part of the section is devoted to introducing canonical families of polynomials on $\sg$, while in the last part we collect some results on the zero sets of continuous functions on $\sg$ that are interesting in their own right.

\subsection{Analysis on \texorpdfstring{$\sg$}{SG}}\label{sec:sub21} 
Let $V_0 = \qty{q_0, q_1, q_2} \in \mathbb{R}^2$, where $q_0:=\qty(\frac 12,\frac{\sqrt 3}2)$, $q_1:=\qty(0,0)$, $q_2:= \qty(1,0)$ and $F_i\qty(x) := \frac{1}{2}\qty(x-q_i) +q_i$ for $i=0,1,2$. Then $\sg$ is the unique nonempty compact set in $\RR^2$ satisfying 
\begin{align}
    \sg = \bigcup_{i=0}^{2}F_{i}\qty(\sg).
\end{align}

Let
 $\omega=\qty(\omega_1,\omega_2,\dots\omega_m)\in\qty{0,1,2}^m$ to be \emph{a word of length} $|\omega|=m$, and  set $F_\omega:=F_{\omega_1} \circ  F_{\omega_2} \cdots F_{\omega_m}$. We call $F_w\qty(\sg)$ an \textit{m-cell} for $|\omega|=m$. For $m\geq 0$, let $V_m = \bigcup_{|w|=m}F_{w}\qty(V_0)$.  Define a sequence of finite graphs $\Gamma_m$ with vertices in $V_m$, and edges generated by the corresponding Euclidean embedding of $V_m$ in $\RR^2$.  In particular, we say that two vertices $x,$ and $y\in V_m$ are neighbor if they lie in the same $m-$cell, in which case we write
 $y \underset{m}{\sim} x$.  Note that the graph $\Gamma_{m+1}$ is a refinement of the graph $\Gamma_m$, i.e., $\Gamma_{m+1}=\cup_{i=0}^2 F_i(\Gamma_{m})$, and that $\sg$ is the limit as $m$ goes to infinity of this sequence of graphs; see Figure~\ref{fig: Graph Approxs to SG}. In the sequel we let  $V^*:=\cup_m V_m$ be  the set of all \textit{vertices}, we term $V_0$ the \textit{boundary points}, and $V^*\setminus V_0$ the set of \textit{junction points}. We refer to \cite{S} for details. Observe also that that, 
 \begin{align}
  \sg = \overline{\bigcup_{m=1}^{\infty}\bigcup_{|w|=m}F_w\qty(V_0)}.
\end{align}

\begin{figure}[h]
    \begin{minipage}{0.45\textwidth}
    \centering
    \includegraphics[width=\textwidth]{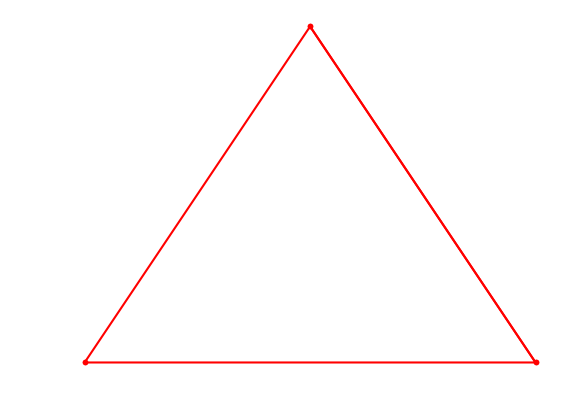}\\\hspace{0.2in}$\Gamma_0$
    \end{minipage}
    \begin{minipage}{0.45\textwidth}
    \centering
    \includegraphics[width=\textwidth]{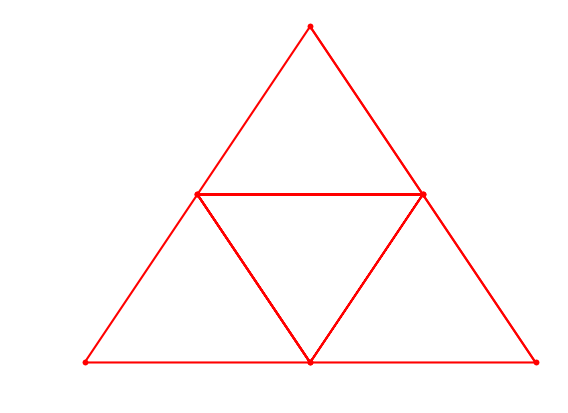}\\\hspace{0.2in}$\Gamma_1$
    \end{minipage}
    \begin{minipage}{0.45\textwidth}
    \centering
    \includegraphics[width=\textwidth]{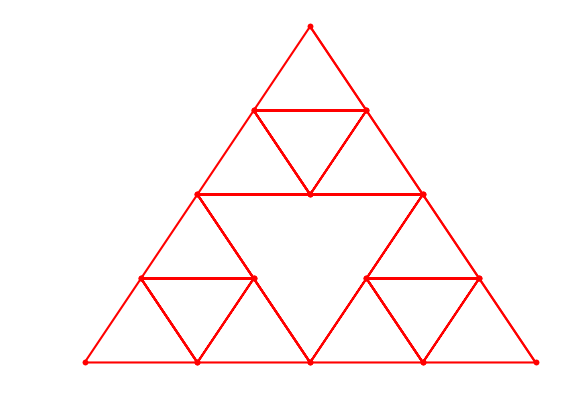}\\\hspace{0.2in}$\Gamma_2$
    \end{minipage}
    \begin{minipage}{0.45\textwidth}
    \centering
    \includegraphics[width=\textwidth]{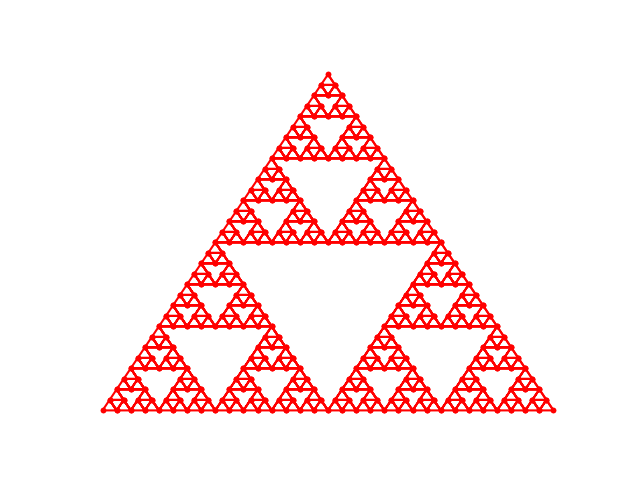}\\\hspace{0.2in}$\Gamma_5$
    \end{minipage}
   \caption{Graph Approximations $\Gamma_m$ to $\sg$ for $m=0,1,2,5$. }
\label{fig: Graph Approxs to SG}
\end{figure}

In this paper, we consider only real-valued functions defined on $\sg$. Let $u,v$ be real-valued functions on $\sg$ and set $E_m\qty(u,v):=\sum\limits_{x\underset{m}{\sim}y} \qty(u\qty(x)-u\qty(y))\qty(v\qty(x)-v\qty(y))$. Observe that $E_m$ can be viewed as an energy functional defined on the finite graph $\Gamma_m$. The \textit{energy} of $u$ and $v$ is defined as $\mathscr{E}\qty(u,v):=\lim\limits_{m\to \infty}\qty(\frac 35)^{-m}E_m(u,v)$. We say $u\in \text{dom}\ \mathscr{E}$ if $\mathscr{E}(u,u)$ exists and is finite. In this case,  all such functions $u$ of finite energy are $1/2$ - H\"older continuous in the effective resistance metric, e.g., see \cite{S}. 
 
We will assume that $\sg$ is equipped with its \textit{standard self-similar measure}, denoted by $\mu$ and which assigns the measure $1/3^m$ to each $m$-cell. We note that $\mu$ is a regular Borel probability measure on $\sg$.

Let $u\in \text{dom}\ \mathscr{E}$ and $f$ a continuous on $\sg$. Then we say $u\in$ dom $\Delta_\mu$ with $\Delta_\mu u=f$ 
if 
\begin{align}
    \mathscr{E}\qty(u,v)=-\int_{SG}fv\dd{\mu} \label{eq: weak formulation of lap}
\end{align}
for any $v\in \text{dom} \ \mathscr{E}$ vanishing on the boundary. Equation \eqref{eq: weak formulation of lap} is termed the \textit{weak formulation} of the Laplacian. Furthermore,  there exists an explicit formula for $\Delta_{\mu}u$ at any junction point $x$, given by $\Delta_{\mu}u\qty(x) = \frac{3}{2}\lim\limits_{m \to \infty}5^m\Delta_mu\qty(x)$, where $\Delta_mu\qty(x) := \sum\limits_{y \underset{m}{\sim} x}\qty(u\qty(y) - u\qty(x))$. Observe that $\Delta_m$ is the graph Laplacian on the finite graph $\Gamma_m$ whose associated quadratic form is the graph energy $E_m$. 
We note that due to the absence of a Leibniz-type rule for $\Delta_{\mu}$, $\text{dom}_{\mu}$ is not an algebra under  pointwise multiplication. In fact, it is known if $u \in \text{dom}\Delta_{\mu}$ then $u^2 \notin \text{dom}\Delta_{\mu}$ \cite{S}. 

The Poisson problem on $\sg$ can be uniquely solved via the $\textit{Green's}$ $\textit{function}$ defined on $\sg$ by 
\begin{equation}
    G\qty(x,y) = \sum\limits_{j=1}^{\infty}\sum\limits_{i=1}^{\infty}\lambda_{i}^{-1}\psi_{i}\qty(x)\psi_{i}\qty(y)\label{eq: explicit formula for green's function},
\end{equation}
where  $\{\lambda_i\}$ is the sequence of eigenvalues of $-\Delta_\mu$ and $\qty{\psi_i}$ is the corresponding set of orthonormal eigenfunctions. It follows that $$ -\Delta_\mu u = f, u|_{V_0} = 0 \iff u\qty(x) = \int_{\sg}G\qty(x,y)f\qty(y)\,\dd{\mu}\qty(y)$$ for any continuous function $f$.

From here on, we will routinely use the Green operator $\mathcal{G}$ whose action on $f \in \text{dom}\Delta_{\mu}$ is given by 
\begin{align}
    \mathcal{G}\qty(f)\qty(x) = -\int_{SG}G\qty(x,y)f\qty(y)\,\dd{\mu}\qty(y) \label{eq: Green's operator explicit formula}.
\end{align}   
Lastly, we will need two notions of derivatives required for the construction of polynomials on $SG$. The \textit{normal derivative} $\partial_nu\qty(q_i)$ of $u$ at $q_i$ is given by
\begin{align}
    \partial_nu\qty(q_i) = \lim\limits_{m \to \infty}\qty(\frac{5}{3})^m\qty(2u\qty(q_i) - u\qty(F^{m}_i q_{i+1}) - u\qty(F^{m}_i q_{i-1})),
\end{align}
and the \textit{tangential derivative} $\partial_Tu\qty(q_i)$ of $u$ at $q_i$ is given by
\begin{align}
    \partial_{T}u\qty(q_i) = \lim\limits_{m\to\infty}5^m\qty(u\qty(F^{m}_i q_{i+1}) - u\qty(F^{m}_iq_{i-1})).
\end{align}


\subsection{Polynomials on \texorpdfstring{$\sg$}{SG}}\label{sec:polynomials}
We will develop a theory of SOPs based on the polynomials introduced in the $\sg$ setting in \cite{BGSKSY, NSTY}. Recall that for any nonnegative integer $j\geq 0$, $f: SG \mapsto \mathbb{R}$ is a said to be a \textit{j-degree polynomial} if and only if $\Delta_\mu^{j+1}f  =  0$ (but $\Delta_\mu^{j}f\neq  0$). 
In other words, $f$ is $j$-harmonic but not ($j-1$)-harmonic. We denote the space of $j$-harmonic functions as $\mathcal H_j$. The polynomials $P_{j,k}^{\qty(l)}$ given below form a basis for the  $3\qty(j+1)$ dimensional space $\mathcal{H}_j$, and can be viewed as the analogs of $\frac{x^j}{j!}$.  Consequently, they are termed \textit{monomials} and the basis $\qty{P_{j,k}^{\qty(l)}}$ is termed the \textit{monomial basis}. The monomials are characterized by the following identities: 
\begin{gather*}
    \begin{cases}
    \Delta^n_\mu P_{j,k}^{\qty(l)}\qty(q_l)&= \delta_{nj}\delta_{k1},\\ \Delta^n_\mu\partial_nP^{\qty(l)}_{j,k}\qty(q_l)&= \delta_{nj}\delta_{k2},\\ 
    \Delta^n_\mu\partial_TP^{\qty(l)}_{j,k}\qty(q_l)&= \delta_{nj}\delta_{k3}.
    \end{cases}
    \end{gather*}

Note that $P^{\qty(l_1)}_{j, k} = R\circ P^{\qty(l_2)}_{j, k}$ where $R$ is the rotation in dihedral group $D_3$ that takes $l_1$ to $l_2$. As a consequence, we proceed by fixing $l = 0$. For simplicity, we will denote $P_{j,k}^{\qty(0)}$ by $P_{j,k}$. Furthermore, the monomials split into three families parametrized by $k$. The families corresponding to $k=1$ and $2$ are symmetric across the line joining $q_0$ and $F_1q_2$, while the $k=3$ family is anti-symmetric across this line. Furthermore, the $k=1$ and $3$ families behave like the even and odd monomials on $\RR$ respectively. The monomials satisfy the following scaling properties from \cite[Equations 2.4-2.6]{NSTY}: 
\begin{gather*}
\begin{cases}
P_{j,1}\qty(F_0^m\qty(x))&=5^{-jm}P_{j,1}\qty(x),\\
P_{j,2}\qty(F_0^m\qty(x))&=\qty(\frac 35)^m5^{-jm}P_{j,2}\qty(x),\\
P_{j,3}\qty(F_0^m\qty(x))&=5^{-\qty(j+1)m}P_{j,3}\qty(x).
\end{cases}
\end{gather*}
For $j\geq 0$, let  $\alpha_j, \beta_j, \gamma_j, \eta_j$ be given by 
\begin{align}
    \alpha_j = P_{j,1}\qty(q_1), \beta_j = P_{j,2}\qty(q_1), \gamma_j = P_{j,3}\qty(q_1), \eta_j = \partial_nP_{j,1}\qty(q_1)
\end{align}
The following recursion relations were proved in \cite[Theorem 2.3, Theorem 2.12]{NSTY}.
\begin{lemma}\label{lem:recabg}
With the initial data $\alpha_0 = 1,\ \alpha_1 = \frac16, \ \beta_0 = -\frac12, \ \eta_0 = 0, \ \partial_nP_{0,2}\qty(q_1)=\partial_n P_{0,2}\qty(q_2) = -\frac12$, we have
\begin{gather}\label{eq: abg}
    \begin{cases}
    \alpha_j = \frac{4}{5^j - 5}\sum\limits_{l=1}^{j-1}\alpha_{\qty(j-l)}\alpha_l, & j \geq 2\\
    \beta_j = \frac{2}{15\qty(5^j - 1)}\sum\limits_{l=0}^{j-1}\qty(3\cdot 5^{j-l} - 5^{l+1} + 6)\alpha_{\qty(j-l)}\beta_l, & j \geq 1\\
    \gamma_j = 3\alpha_{j+1}, & j \geq 0\\
    \eta_j = \frac{5^j +1}{2}\alpha_j + 2\sum\limits_{l=0}^{j-1}\eta_l\beta_{\qty(j-l)}, & j \geq 1.
    \end{cases}
\end{gather}
In addition,   
$\partial_nP_{j,2}\qty(q_1) =\partial_n P_{j,2}\qty(q_2)= -\alpha_j$ for $j \ge 1$ and $\partial_nP_{j,3}\qty(q_1) =-\partial_nP_{j,3}\qty(q_2)= 3\eta_{j+1}$ for $j \geq 0$. 
\end{lemma}

\subsection{Zeroes of Polynomials}\label{zerospoly}
This section deals with zeros of functions defined on $SG$. While not directly related to the topics of orthogonal polynomials on $SG$, these results are interesting in their own right and could lead to more investigations on topics such as the nodal sets of polynomials and eigenfunctions on $SG$. The first result uses the topological structure of $SG$ to derive a property of zeros for continuous functions.
\begin{proposition}\label{pr:zeros}
Let $f$ be a continuous function defined on $SG$ with finitely many zeros. Suppose that $Z=\qty{z_k}_{k=1}^m$ is the set of the zeros of  $f$, and set $Z_0=Z\cap V^*$. Then for any connected component $D$ in $SG\setminus Z_0$, either $f\ge0$ on $D$ or $f\le0$ on $D$.
\end{proposition}

\begin{proof}
Suppose $f$ has finitely many zeros and $f\qty(z_1)>0$, $f\qty(z_2)<0$ with $z_1,\,z_2\in D$. Define $$s=\min\limits_{\substack{x\neq y\in SG\\f(x)\neq f(y)\neq 0}}|f(x)-f(y)|,$$ and choose $2^{-n}<s$. By considering small enough neighborhoods of the two points, we can assume without loss of generality that $z_1$ and $z_2$ are both junction points in $V_m$ and $m>n$. Then there exists a simple curve $L:\qty[0,1]\rightarrow SG$ from $z_1$ to $z_2$ along edges in $\Gamma_m$ that lie in $D$ with constant speed. If $f\circ L\equiv0$ on $\qty[\frac13,\, \frac23]$, then f has infinitely many zeros and contradiction arises. Otherwise there exists $t_0\in \qty[\frac13,\, \frac23]$ and  $f\circ L\qty(t_0)\neq 0$. Then we can find $z_1'$, $z_2'$ on the curve $L$ such that $f\qty(z_1')f\qty(z_2')<0$ and $d\qty(z_1',z_2')\le\frac23 d\qty(z_1, z_2)$, where $d$ denotes the distance along $L$ and is bounded below by the Euclidean metric. Then denote the new $z_i'$ by $z_i$ and continue this process until $d\qty(z_1,z_2)<2^{-m-1}$. Then $z_1$ and $z_2$ lie on edges of the same or adjacent $m$-cells. In either case we can find two curves from $z_1$ to $z_2$ that intersect at most at $z_1$, $z_2$ and a junction point $z$ that lies in $L$, but none of them belongs to $Z$. Hence by $IVT$ we get two zeros of $f$ with distance $<2\cdot2^{-m}\le 2^{-n}<s$, which is a contradiction.
\end{proof}

Next we consider the zeros set of analytic functions on $SG$. First recall the following definition. 
\begin{definition}\label{def:analytic} An entire analytic function is a function defined by power series $\sum\limits^3_{k=1}\sum\limits^\infty_{j=0}c_{j, k}P_{j,k}\qty(x)$ such that $|c_{j, k}|=O\qty(R^j)$ for some $R<\lambda_2$, where $\lambda_2$ is the second nonzero Neumann eigenvalue of the Laplacian. 
\end{definition}
\begin{proposition}\label{pr:analytic_zero}
Let $f=\sum\limits^3_{k=1}\sum\limits^\infty_{j=0}c_{j, k}P_{j,k}\qty(x)$ be a nonzero entire analytic function, then the zero set $Z$ of $f$ is compact and nowhere dense in $SG$.
\end{proposition}
\begin{proof}
Suppose $Z$ has an interior point, then $Z$ must contain $F_\omega\qty(SG)$ for some $\omega$, i.e. its local power series expansion in $F_\omega\qty(SG)$ is identically zero. Note that it has a unique extension to an entire analytic function in $SG$ by \cite[Theorem 3.6]{NSTY},  $f$ must be identically zero in $SG$.
\end{proof}

We end this section by  describing the behavior of  zeros near the point $q_0$ for an entire analytic function $f$. In addition, under the more restrictive assumption that the following inequality
\begin{gather}P_{j,1}\qty(x)>0 \;\forall x\neq q_0,\;\forall j>0 \label{conjecture}
\end{gather} which was conjectured in \cite{NSTY}, is true, then we can prove a slightly stronger result.

\begin{theorem}\label{th:zeros}
Suppose $f\qty(x)=\sum\limits^\infty_{j=t_1}c_{j, 1}P_{j,1}\qty(x)+\sum\limits^\infty_{j=t_2}c_{j, 2}P_{j,2}\qty(x)+\sum\limits^\infty_{j=t_3}c_{j, 3}P_{j,3}\qty(x)$ where $c_{t_1,1}$, $c_{t_2,2}$ and $c_{t_3,3}$ are nonzero and denote by $Z$ the  zero set of $f$. If $ t_3<t_1-1$ and $t_3<t_2$, then $f$ has infinitely many zeros that have a limit point $q_0$. 

Furthermore,  suppose that  conjecture $\eqref{conjecture}$ is true, $t_1\le t_2$ and $t_1\le t_3$. Then $q_0$ has a neighborhood $U$ such that $Z\cap U\subset\qty{q_0}$.
\end{theorem}
\begin{proof}
For the first claim, note $\gamma_{t_3}>0$, and WLOG assume $c_{j, 3}>0$. Then for $n$ large enough, we have
\begin{align}
f\qty(F_0^n\qty(q_1))&\ge c_{t_3,3}5^{-\qty(t_3+1)n}\gamma_{t_3}-\sum\limits^\infty_{j=t_1}5^{-jn}|c_{j, 1}|\|P_{j,1}\|_{L^\infty}\\
             &\;-\sum\limits^\infty_{j=t_2}\qty(\frac 35)^n5^{-jn}|c_{j, 2}|\|P_{j,2}\|_{L^\infty}-\sum\limits^\infty_{j=t_3+1}5^{-\qty(j+1)n}|c_{j, 3}|\|P_{j,3}\|_{L^\infty} \\
             &\ge c_{t_3,3}5^{-\qty(t_3+1)n}\gamma_{t_3}-C\qty(5^{-t_1n}+\qty(\frac 35)^n5^{-t_2n}+5^{-\qty(t_3+2)n})>0.
\end{align}
Similarly, and using the fact that $P_{j,3}$ is  anti-symmetric, we have $f\qty(F_0^n\qty(q_2))<0$ when $n$ is large enough. Hence for large n, there always exists a zero of $f$ on the straight line connecting $F_0^n\qty(q_2)$ and $F_0^n\qty(q_1)$. 

As for the second claim, note that if conjecture $\eqref{conjecture}$ is true, then there exists $c>0$ such that $P_{t_1,1}>c$ on $F_1\qty(SG)\cup F_2\qty(SG)$. Hence by similar argument,
\begin{align}
f\qty(F_0^n\qty(x))\ge 5^{-t_1n}c_{t_1,1}c-C\qty(5^{-\qty(t_1+1)n}+\qty(\frac 35)^n5^{-t_2n}+5^{-\qty(t_3+1)n})>0
\end{align}
for any $x\in F_1\qty(SG)\cup F_2\qty(SG)$, when $n$ is large enough.
\end{proof}
\begin{remark}
Using this method, we can express an entire analytic function in terms of local power series expansion and study the local behavior of zeros near a junction point.
\end{remark}

\section{Sobolev-Legendre Orthogonal Polynomials on \texorpdfstring{$\sg$}{SG}}\label{sec:sopsg}
In this section, we will consider the Sobolev inner product and their corresponding orthogonal polynomials. More specifically, given $\chi>0$ and recalling that  $\mu$ denotes the standard self-similar probability measure on $SG$, we consider the Sobolev inner product 
\begin{align}
  \ip{f}{g}_S=\int_{SG}f\qty(x)g\qty(x)\, \dd{\mu}\qty(x)+\chi \int_{SG} \Delta f\qty(x)\Delta g\qty(x)\, \dd{\mu}\qty(x). 
\end{align}\label{eq: basic sobolev inner product}
We denote by $W^{1,2}\qty(SG)$ the Hilbert space on $SG$ corresponding to the Sobolev inner product defined above. $W^{1,2}$ is the Sobolev space of $L^2$  functions whose Laplacian belongs to $L^2$, we refer to \cite{StriFctSp} for more on function spaces on fractals. The orthogonal polynomials we will construct will allow one to solve the following optimization problem
$$\operatorname{argmin}\|f-g_n\|_{S},$$ where $f\in W^{1,2}\qty(SG)$ and $g_n$ is a polynomial of degree at most $n$ on $SG$.

\subsection{General properties}\label{sec3.1}
The  Sobolev-Legendre OPs exhibit different properties depending on the family of monomials they are generated from. More precisely, the polynomials corresponding to the families $k=2$ or $3$ have the same properties but differ from the $k=1$ family.
In this section we focus on properties that are common to the three families while, Section~\ref{sec3.2} focuses on the cases $k=2$ or $3$, and we defer the $k=1$ case to Section~\ref{sec: SOP wrt k=1}. Finally, we consider Sobolev-Legendre OPs with inner products involving higher powers of the Laplacian in Section~\ref{subsec:higherordersop}.

\begin{definition}\label{def:basicsop}
For fixed $k=1,2$ or $3$, we define the  Sobolev-Legendre orthogonal polynomials (with respect to $q_0$) to be the set $\qty{s_{n,k}\qty(x;\chi)}_{n=0}^{\infty}$ of orthogonal polynomials obtained by applying the Gram-Schmidt to the sequence of monomials $\qty{P_{n,k}}_{n=0}^\infty$, that is, $$s_{n,k}\qty(x; \chi)=P_{n,k}\qty(x)-\sum\limits_{\ell=0}^{n-1}\xi_\ell^2\ip{P_{n, k}}{s_{\ell, k}}_S\, s_{\ell, k}\qty(x).$$ Furthermore, there exists a set of coefficients $\qty{z_{\ell, n}}_{\ell=0}^n$ with $z_{n, n}=1$ such that $$s_{n,k}\qty(x; \chi)=P_{n,k}\qty(x)-\sum\limits_{\ell=0}^{n-1}\xi_\ell^2\ip{P_{n, k}}{s_{\ell, k}}_S\, s_{\ell, k}\qty(x)=P_{n,k}\qty(x)+\sum\limits_{\ell=0}^{n-1}z_{\ell, n}P_{\ell, k}\qty(x).$$ It follows that 
$$\ip{s_{\ell, n}}{s_{k,n}}_S=\xi_\ell^{-2}\delta_{\ell k}\quad \text{where}\quad \|s_{\ell, k}\qty(\cdot; \chi)\|_S^2=\xi_\ell^{-2}.$$

When there is no confusion about $k$ and $\chi$ we will simply write $\qty{s_{n}}_{n=0}^{\infty}$. The corresponding orthonormal polynomials will be denoted $\qty{S_{n, k}\qty(x; \chi)}_{n\geq 0}$ or $\qty{S_{n}}_{n\geq 0}$ when there is no confusion.

\end{definition}

\begin{remark}\label{rem:defLegendreop}
We recall that the Legendre OP (with respect to $q_0$) introduced in \cite{OST} are denoted by $\qty{p_{n,k}\qty(x)}_{n=0}^{\infty}$ ($\qty{p_{n}(x)}_{n=0}^{\infty}$ when there is no ambiguity) and the corresponding orthonormal Legendre OP are denoted by $\qty{Q_{n, k}}_{n\geq 0}$ ($\qty{Q_{n}}_{n\geq 0}$ when there is no ambiguity). 

In addition, 
$$p_{n,k}\qty(x)=P_{n,k}-\sum\limits_{\ell=0}^{n-1}d_\ell^2\ip{P_{n,k}}{p_{\ell, k}}_2\, p_{\ell, k}\qty(x)=P_{n,k}+\sum\limits_{\ell=0}^{n-1}\omega_{\ell, n}P_{\ell, k}\qty(x),$$ where $\qty{\omega_{\ell, n}}_{\ell=0}^n$ is a set of coefficients with $\omega_{n, n}=1$, and 
$$\ip{p_{\ell, k}}{f_{\ell', k}}=d_\ell^{-2}\delta_{\ell \ell'}\quad \text{where}\quad \|p_{\ell, k}\|^2=d_\ell^{-2}.$$
\end{remark}

We begin by proving some estimates on the norms of the polynomials $s_{n,k}$.

\begin{theorem}\label{thm:normestimate}
For $k=1, 2, 3$, and $n\geq 0$ we have the following estimates. 
$$\|p_{n,k}
\|_2^2<\|s_{n,k}\|_2^2<\|s_{n,k}\|_S^2< \|P_{n,k}\|_S^2=\|P_{n,k}\|_2^2+\chi \|_{n-1, k}\|_2^2.$$ In addition, for any $0<r<\infty$, there exist positive constants $c_1, c_r$ such that 
$$\|s_{n,k}\|_S<\qty(1+\chi)\qty(c_1 \qty(\qty(n-1)!)^{-\log5/\log 2} + c_rr^{-n}).$$ Consequently, $\lim\limits_{n\to \infty}\|s_{n,k}\|_S= \lim\limits_{n\to \infty}\|s_{n,k}\|_2 =0.$

\end{theorem}

\begin{proof}
We note that $\|s_n\|_S^2=\|s_n\|_2^2+\chi \|\lap s_n\|_2^2$. Moreover, it follows from the definition that $s_n=P_n-\sum\limits_{\ell=0}^{n-1}\xi_{\ell}^{2}\ip{P_n}{s_\ell}_S\, s_\ell.$ Consequently, 
$$\|s_{n}\|_S^2= \|P_n\|_S^2-\sum\limits_{\ell=0}^{n-1}\xi_{\ell}^2|\ip{P_n}{s_\ell}_S|^2<\|P_n\|_S^2=\|P_{n,k}\|_2^2+\chi \|P_{n-1,k}\|_2^2.$$  The rest of the proof follows from \cite[Theorem 3.1]{OST}.

\end{proof}

 The next two technical results will be useful in deriving a recurrence relation between the Sobolev-Legendre  and Legendre orthogonal polynomials.


\begin{lemma}\label{Lemma: Expansions for f}
For each $k\in \qty{1,2,3}$ fixed, and any integer $t\geq 1$, let $$f_{t+1,k} := -\int_{SG}G\qty(x,y)p_{t,k}\qty(y)dy. $$  Then for any $t \geq 2$, we have the following statements hold.
\begin{equation}\label{eq: dnftfor2or3atq1}
\begin{cases}
 \partial_n f_{t,k}\qty(q_1)=\partial_n f_{t,k}\qty(q_2)=0, \, \, \text{when}\, k=2, \, \text{or}\,\,  3 \\
  \partial_n f_{t,k}\qty(q_0)=\int_{SG}p_{t-1,k}\qty(x)\,\dd{\mu}\qty(x), \, \, \text{when}\, \, k=2 \\
   \partial_n f_{t,k}\qty(q_0)+2\partial_n f_{t,k}\qty(q_1)= 0, \, \, \text{when}\,\,  k=1 \\
   \partial_n f_{t,k}\qty(q_1)+\int_{SG}p_{t-1, k}\qty(x)P_{0,2}\qty(x)\,\dd{\mu}\qty(x) = 0, \, \, \text{when}\,\,  k=1
\end{cases}    
\end{equation}

\end{lemma}

\begin{proof}
We will use  the Gauss-Green formula given by $$\int _{SG}f\lap g-\int_{SG}g\lap f
\dd{\mu}=\sum\limits_{l=0}^{2}f\qty(q_l)\partial_n g\qty(q_l)-g\qty(q_l)\partial_n f\qty(q_l).$$ 

Take $f=f_{t,k}$ and set
$g=p_{0,k}$. For this choice of $g$, the left side of the Gauss-Green formula always vanishes because $\Delta g=0$ and $\inner{g}{\Delta f_t}_{L^2} = \inner{p_{0,k}}{p_{t-1, k}}_{L^2} = 0$. Furthermore, the first term on the right hand side vanishes because $f_t$ vanishes on the boundary. As for the second term on the right, when $k=2$ or $3$, $g\qty(q_0)\partial_n f\qty(q_0) = 0$ because for these values of $k$, $g$ vanishes at $q_0$. 

When $k = 2$ we have due to symmetry that $\partial_n f\qty(q_1) = \partial_n f\qty(q_2)$ and $g\qty(q_1) = g\qty(q_2) = -1/2$; as a result the right side of the Gauss-Green formula reads $\frac{-1}{2}\qty(\partial_n f\qty(q_2) + \partial_n f\qty(q_1)) = 0$, implying the first equation in~\eqref{eq: dnftfor2or3atq1} for $k=2$. 
The case  $k=3$, is treated similarly using the fact that 
 $g$ and $f$ are both anti-symmetric.
 
 The second equation in~\eqref{eq: dnftfor2or3atq1} is obtained by setting $g:= p_{0,1}$ in the Gauss-Green formula.

On the other hand, when $k=1$, let $g \equiv 1$ and $f=f_t$. It follows that the left hand side of the Gauss-Green formula vanish, while  its right hand side reduces to $\partial_n f_t\qty(q_0)+2\partial_n f_t\qty(q_1)=0$, which is precisely the third equation in ~\eqref{eq: dnftfor2or3atq1}. Finally, if we set $g:= p_{0,2}$  then we get the last equation in~\eqref{eq: dnftfor2or3atq1}.
\end{proof}

The following lemma is motivated by \cite[Theorem 3.5]{OST} and allows us to recursively compute the polynomial $f_n$ defined in Lemma~\ref{Lemma: Expansions for f}. We recall that the sequences  $\alpha_i, \beta_i,$ and $\gamma_i$ were defined in  Lemma~\ref{eq: abg}.

\begin{lemma}\label{Lemma: Expansions for zeta}
Suppose $k= 1,2$ or $3$. For any $j\geq 0$, 
$$p_j:=p_{j,k}=\sum\limits_{l=0}^j\omega_{j,l}P_{l,k},\,\, \zeta_{j, 1} = 2\sum\limits_{l=0}^j\omega_{j,l}\alpha_{l+1}, \,\, \zeta_{j, 2} =2\sum\limits_{l=0}^j\omega_{j,l}\beta_{l+1}, \,\, \text{and}\, \,
\zeta_{j, 3} = -2\sum\limits_{l=0}^j\omega_{j,l}\gamma_{l+1}.$$
Then
\begin{equation} \label{ftfor12}
\begin{cases}
 f_{j+1,k}=\zeta_{j, k}P_{0,2}+\sum\limits_{l=0}^j \omega_{j,l}P_{l+1,k},\, \, \text{when}\, \, k=1\, \, \text{or}\, \, 2,\\
  f_{j+1,3}=\zeta_{j, 3}P_{0,3}+\sum\limits_{l=0}^j \omega_{j,l}P_{l+1,3},\, \, \text{when}\, \, k=3.
\end{cases}
\end{equation}
\end{lemma}

\begin{proof} Fix $k=1,2,$ or $3$. 
Note that $$
f_{j+1}\qty(x):=f_{j+1, k}\qty(x)=-\int \;G\qty(x,y)p_j\qty(y)\dd{\mu}\qty(y)=-\sum\limits^j_{l=0}\omega_{j,l}\int\; G\qty(x,y)P_{l,k}\qty(y)\dd{\mu}\qty(y). $$ 
However,
\begin{align}
    -\int\; G\qty(x,y)P_{l,k}\qty(y)\dd{\mu}\qty(y)= \begin{cases}P_{l+1,k}+2\alpha_{l+1}P_{0,2},&k=1,\\
P_{l+1,k}+2\beta_{l+1}P_{0,2},&k=2,\\
P_{l+1,k}-2\gamma_{l+1}P_{0,3},&k=3.
\end{cases}
\end{align}
\end{proof}

\begin{remark}
When $k=1$ in~\eqref{ftfor12}, we see that $f_{j+1, 1}$ does not belong to the same family of polynomials generated from the set of monomials with $k=1$. This fact will play a subtle  role in the numerical experiments we describe later. 
\end{remark}

\subsection{Sobolev Orthogonal Polynomials with respect to \texorpdfstring{$k=2,3$}{k=2,3}}\label{sec3.2}

In this section we investigate the  Sobolev-Legendre OPs corresponding to the families $k=2$ and $3$, as they behave similarly.

\begin{lemma}\label{Lemma: Family 2 and 3 is self contained on Green's, Family 1 isn't}
Fix $k=2$ or $3$. Let $C$ be a polynomial in $\operatorname{span}\qty{P_{n,k}}_{n=0}^{\infty}$ with $\operatorname{deg}\qty(C) = J$. Define the function $f$ on $SG$ by $f\qty(x):= -\int_{SG} G\qty(x,y)C\qty(y)dy$. Then, $f$ is also in $\operatorname{span}\qty{P_{n,k}}_{n=0}^{\infty}$ with $\operatorname{deg}\qty(f) = J+1$. In other words $f$ remains in the span of the same family of monomials as $C$. 
\end{lemma}

\begin{proof}
Observe that by distributing the Green's operator in the polynomial we can write $f$ as
$$f\qty(x)= \sum\limits_{j=0}^{J}c_jP_{j+1, k}\qty(x) - \sum\limits_{j=0}^{J} c_j H_j\qty(x),$$ 
where $H_{j}$ is a harmonic function with the same boundary values as $P_{j+1, k}$. As such we can write $H_j\qty(x)=\sum\limits_{l=1}^3d^{\qty(j)}_lP_{0,l}\qty(x)$.
Because $f$ vanishes at $q_0$, we conclude that the coefficient of $P_{0,1}$ in the above formula is $0$. 

Next, when $k=2$, $C$ and $f$ are both symmetric, hence the coefficient of $P_{0,3}$ is zero. Similarly, when $k=3$,  $C$ and $f$ are both  antisymmetric, hence the coefficient of $P_{0,2}$ is zero. 
\end{proof}

For $k=2,$ or $3$ fixed, and given the Legendre OPs $\qty{p_{n,k}}$, the polynomial $f_{n+1}\qty(x)=-\int_{SG}G\qty(x,y)p_n\qty(y)\, \dd{\mu}\qty(y)$ is of degree $n+1$. When expressed in terms of the Legendre OPs as in~\cite[Theorem 3.2]{OST}, one obtains a substitute of the three-term recursion formula. Instead, if we express $f_{n+1}$ in terms of the Sobolev-Legendre OPs, we obtain the following result. This recursion will be used when we plot the Sobolev-Legendre OPs in Section~\ref{subsec: numericsSLOPs}.

\begin{theorem}\label{th:k23recurrence} Fix $k=2$ or $3$, and let $\qty{s_n}_{n\geq-1}$ be the family of Sobolev OP, where  $s_{-1}:=0$. Let $f_0\qty(x):=f_{0,k}\qty(x)=0$ and for $n\geq 0$ let  $$f_{n+1}\qty(x):=f_{n+1,k}\qty(x)= -\int_{SG}G\qty(x,y)p_{n}\qty(y)\, \dd{\mu}\qty(y),$$ where $\qty{p_n}_{n=0}^\infty$ is the corresponding set of monic Legendre OPs.  Then the  Sobolev OP satisfy the following recurrence relation:
\begin{equation}
    s_{n+1} + a_{n}s_{n}+\tilde{b}_{n}s_{n-1} = f_{n+1}\quad n\geq 0,\label{3-term recurrence}
\end{equation}
where $$ a_{n} = \frac{\inner{f_{n+1}}{s_{n}}_S}{\|s_n\|_S^2},\qquad \tilde{b}_{n} = \frac{\inner{f_{n+1}}{s_{n-1}}_S}{\|s_{n-1}\|_S^2}.$$
\end{theorem}

\begin{proof} Fix $n\geq 0$. Because $f_{n+1}$ is an $n+1$ degree  polynomial, we have
 $f_{n+1} = \sum\limits_{j=0}^{n+1}a_{j}s_{j},$
where $a_j = \tfrac{\langle f_{n+1}, s_j \rangle_{S}}{\|s_j\|_S^2}$. Note that if  $j < n-1$, 
\begin{align}
    \langle f_{n+1}, p_j \rangle_{S} &= \int_{SG}f_{n+1}p_{j}\,\dd{\mu} + \chi\int_{SG}\lap f_{n+1} \lap p_j \dd{\mu}\\
    &= -\int_{SG}\int_{SG}G\qty(x,y)p_{n}\qty(y)\,\dd{\mu}\qty(y)p_{j}\qty(x)\,\dd{\mu}\qty(x) + \chi\int_{SG}p_{n}\, \lap p_{j}\,\dd{\mu}\\
    &= -\int_{SG}p_{n}\qty(y)\int_{SG}G\qty(x,y)p_{j}\qty(x)\,\dd{\mu}\qty(x)\,\dd{\mu}\qty(y) + \chi\int_{SG}p_{n}\lap p_{j}\,\dd{\mu} \\
    &= \int_{SG}p_{n}\qty(y)f_{j+1}\qty(y) \,\dd{\mu}\qty(y) + \chi\int_{SG}p_{n}\lap p_{j}\,\dd{\mu} \\
    &= 0, 
\end{align}
where we have used the fact  $p_{n}$ is orthogonal to lower order polynomials in the standard inner product. If follows that, for $j<n-1$, $\ip{f_{n+1}}{s_j}_S=0$ as $s_j$ can be written as linear combination of $p_\ell$ where $\ell\leq j< n-1.$ Furthermore, the coefficient $a_{n+1} = 1$ since from Lemma \ref{Lemma: Expansions for f}, $f_{n+1}$ is monic of degree $n+1$. The rest of the coefficients are recovered by projection. 
\end{proof}   
 
 Using the notations of Remark~\ref{rem:defLegendreop} and~\cite[Theorem 3.2]{OST} the following corollary is easily established.
 
 \begin{corollary}\label{corfnsnbn}
 Fix $k=2$ or $3$ and set $s_{-1}\qty(x)=p_{-1}\qty(x)=0$. For $n\geq 0$ set  $f_{n+1}\qty(x)=-\int_{SG}G\qty(x, y)p_{n}\qty(y)\, \dd{\mu}\qty(y)$. The Legendre OPs $\qty{p_{n,k}}_{n=0}^{\infty}$ and the Sobolev-Legendre OPs $\qty{s_{n,k}}_{n=0}^\infty$ satisfy the following relation.
 $$f_{n+1}(x)=s_{n+1}+a_ns_n+\tilde{b}_{n}s_{n-1}=p_{n+1}+b_np_n+c_np_{n-1},$$ where the coefficients $b_n, c_n$ are defined in ~\cite[Theorem 3.2]{OST}.
 \end{corollary}

We next prove that as in the classical case, the Sobolev-Legendre OPs satisfy a second order differential equation involving the Legendre OPs. This result should be compared to \cite[Theorem 3.1]{Meij}. Note that we use the notation of \cite{OST}.

\begin{theorem}\label{th: sobolev ode}
Fix $k=2$ or $3$. Then for each $n\geq 0$, the Sobolev OPs satisfy the following second order differential equation: 
\begin{align}
s_n\qty(x)+\chi \Delta^2 s_n\qty(x)&=\Delta p_{n+1}\qty(x)+a_nd_n^2\xi_n^{-2}\Delta p_n\qty(x)+d_{n-1}^2\xi_{n}^{-2}\Delta p_{n-1}\qty(x) \label{eq: Sobolev ODE} \\
&=p_n\qty(x)+\qty(a_nd_n^2\xi_n^{-2}-b_n)\lap p_n +\qty(d_{n-1}^2\xi_n^{-2}-c_n)\lap p_{n-1},\notag
\end{align}
 where $\qty{p_n}_{n\geq 0}$ are the corresponding Legendre OPs with  $p_{-1}=0,$ $a_n$ is given as in Theorem~\ref{th:k23recurrence}, $d_n^{-2}=\|p_n\|_2^2,$ and  $\xi_{n}^{-2}=\|s_n\|_S^2.$
\end{theorem}

\begin{proof}
For $n\geq 0$, let  $$T_{n+1}\qty(x)=-\int_{SG}G\qty(x, y) s_n\qty(y)\dd{\mu}\qty(y).$$ It is clear that $\Delta T_{n+1}\qty(x)=s_n\qty(x)$ and $T_{n+1}$ is a polynomial of degree $n+1$, with $$T_{n+1}\qty(x)=p_{n+1}\qty(x)+\text{lower\, order\, terms}.$$ Given any polynomial $h$ of degree at most $n-2$, let $g\qty(x)=-\int_{SG}G\qty(x, y)h\qty(y)\, \dd{\mu}\qty(y)$. Since $g$ is a polynomial of degree at most $n-1$, we see that 
\begin{align}
    \ip{s_n}{g}_S&=\int_{SG}s_n\qty(x)g\qty(x)\, \dd{\mu}\qty(x)+\chi\int_{SG}\Delta s_n\qty(x) \Delta g\qty(x)\, \dd{\mu}\qty(x)\\
    &=\int_{SG}s_n\qty(x)g\qty(x)\, \dd{\mu}\qty(x)+\chi\int_{SG}\Delta s_n\qty(x) h\qty(x)\, \dd{\mu}\qty(x)=0.
    \end{align}
Consequently, using the Gauss-Green formula we have
\begin{align}
 0&=\int_{SG}\qty[s_n\qty(x)g\qty(x)+\chi \Delta s_n\qty(x)\Delta g\qty(x)]\dd{\mu}\qty(x)\\
 &=\int_{SG}\qty[T_{n+1}\qty(x)\Delta g\qty(x)+\chi \Delta s_n\qty(x)\Delta g\qty(x)]\dd{\mu}\qty(x) +\sum\limits_{\ell=0}^2 g\qty(q_{\ell})\partial_n T_{n+1}\qty(q_{\ell})-T_{n+1}\qty(q_{\ell})\partial_ng\qty(q_{\ell})\\
 &= \int_{SG}\Delta g\qty(x)\qty[T_{n+1}\qty(x)+\chi \Delta^2 T_{n+1}\qty(x)]\dd{\mu}\qty(x) +\sum\limits_{\ell=0}^2 g\qty(q_{\ell})\partial_n T_{n+1}\qty(q_{\ell})-T_{n+1}\qty(q_{\ell})\partial_ng\qty(q_{\ell}).\\
\end{align}
However, $\sum\limits_{\ell=0}^2 T_{n+1}\qty(q_{\ell})\partial_ng\qty(q_{\ell}) - g\qty(q_{\ell})\partial_n T_{n+1}\qty(q_{\ell})=0$, so 
$$\int_{SG}h\qty(x)\qty[T_{n+1}\qty(x)+\chi \Delta^2 T_{n+1}\qty(x)]\dd{\mu}\qty(x)=0$$ for all polynomial $h$ of degree at most $n-2$. It follows that we can write $$T_{n+1}+\chi \Delta^2T_{n+1} =p_{n+1}+y_np_n+t_np_{n-1}.$$ Now, $$y_n=d^2_n\ip{T_{n+1}+\chi \lap^2T_{n+1}}{p_n}_2=d_n^2\ip{T_{n+1}}{p_n}_2=d_n^{2}\ip{s_n}{f_{n+1}}_2=d_n^2\ip{s_n}{f_{n+1}}_S=a_nd_n^2\xi_n^{-2}.$$
Similarly,
$$t_n=d_{n-1}^2\ip{T_{n+1}+\chi \lap^2T_{n+1}}{p_{n-1}}_2= d_{n-1}^2\ip{f_n}{s_n}_S=d_{n-1}^2\|s_n\|_S^2=d_{n-1}^2\xi_{n}^{-2}.$$ 

Using \cite[Theorem 3.2]{OST} and taking the Laplacian on both sides yield the result.
\end{proof}

  \begin{remark}\label{rem:matrix}
\begin{enumerate}
  \item Note that if we write $s_n\qty(x)=\sum\limits_{\ell=0}^nz_{\ell, n}P_{\ell}\qty(x)$ and recall that $p_n\qty(x)=\sum\limits_{\ell-0}^n\omega_{\ell, n}P_\ell\qty(x)$ with $\omega_{n, n}=z_{n, n}=1$, and  substitute these in~\eqref{eq: Sobolev ODE}, we obtain the following recursive formulas.
\begin{align}
z_{n-1,n}&=\omega_{n, n+1}+a_nd_n^2\xi_n^{-2},\\
z_{n-2,n}&=-\chi+\omega_{n-1,n+1}+a_nd_n^2\xi_n^{-2}\omega_{n-1,n}+d_{n-1}^2\xi_n^{-2},\\
z_{n-3,n}&=-\chi\qty(\omega_{n, n+1}+a_nd_n^2\xi_n^{-2})+ \omega_{n-1,n+1}+a_nd_n^2\xi_n^{-2}\omega_{n-2,n}+d_{n-1}^2\xi_n^{-2}\omega_{n-2,n-1},\\
z_{n-\ell,n}&=-\chi z_{n-\ell +2,n} +\omega_{n-\ell+1, n+1}+a_nd_n^2\xi_n^{-2}\omega_{n-\ell+1,n}+ d_{n-1}^2\xi_n^{-2}\omega_{n-\ell+1, n-1},\quad \ell = 4, 5, \hdots n.
\end{align}

\item 
We can rewrite~\eqref{3-term recurrence} in terms of the following matrix:
  $$\mathcal{A}\vv{s}=\vv{f},$$ where $\mathcal{A}$ is a semi-infinite upper triangular  matrix such that
  \begin{align*}
      \mathcal{A}_{n,n}&=\tilde{b}_n, \,\, \, \mathcal{A}_{n,n+1}=a_n, \,\, \, \mathcal{A}_{n,n+2}=1,\, \, \, \text{and}\, \, \,  \mathcal{A}_{n,m}=0 \, \, \text{otherwise, with}\\
      \vv{s}&=\qty(s_0, s_1, s_2, \hdots )^T,\quad \vv{f}=\qty(f_2,f_3,f_4, \hdots )^T.
  \end{align*}
 
\end{enumerate}
 \end{remark}

We collect below a number of properties of the  Sobolev OPs when  $k=2$ or $3$. 
In particular, the next results give some refined estimates for  $\|s_{n}\|_S$ and the coefficients $a_n$ and $\tilde{b}_n$ in terms of the the norm of Legendre polynomials
and $\chi$. 

\begin{proposition}\label{Proposition: L2, H1, coefficients estimates}
Let  $\qty{a_n}$ and $\qty{\tilde b_n}$ be defined as in \eqref{3-term recurrence}. Then for $n\ge 1$ the following estimates hold.
\begin{equation}\label{SobolevL2H1Estimate}
\begin{cases}
  \|p_n\|_{2}\leq  \|s_{n}\|_{2} \leq \|G\|_{2}\|p_{n-1}\|_{2}, \\
 \qty(1+\chi\|G\|_2^{-2}) \|p_n\|_{2}^2 \leq      \|s_n\|_{S}^2\leq \qty(\|G\|_{2}^2+\chi)\|p_{n-1}\|_{2}^2 ,
   \end{cases}
   \end{equation}
   
 and 
   
\begin{equation} \label{anbbounds}
\begin{cases}
 \qty|a_n|\leq \mini{\|G\|_{2}}{\chi^{-1}\|G\|_{2}^3},\\
   0<\tilde{b}_n\leq \mini{\|G\|_{2}^2}{\chi^{-1}\|G\|_{2}^4}.
   \end{cases}
\end{equation}
where $p_{n}$ is the $n$-th monic Legendre orthogonal polynomial. 

In particular, it follows that  

$$\|s_n\|_{S}^2=O \qty(\chi), \quad 
a_n=O\qty(\chi^{-1}), \quad \text{and}\quad 
\tilde b_n=O\qty(\chi^{-1}).$$ 
Furthermore, 
$$ 
\lim\limits_{\chi \to \infty} \chi \tilde{b}_n=\frac{\|p_{n}\|_{2}^2}{\|p_{n-2}\|_{2}^2}, \quad
\lim\limits_{\chi \to \infty} \chi a_n=\frac{\ip{f_{n+1}}{f_{n}}}{\|p_{n-1}\|_{2}^2},\,  \text{and}\quad 
\lim\limits_{\chi \to \infty} \|\lap s_n\|_2=\|p_{n-1}\|_2$$ 
uniformly in $n$.
\end{proposition} 

\begin{proof}
The lower inequality in the first estimates in~\eqref{SobolevL2H1Estimate} follows from the fact that $p_n$ and $s_n$ are monic polynomials, while the upper estimate follow from the fact that $$\|s_n\|_2^2+\chi \|p_{n-1}\|_2^2\leq\|s_n\|_2^2+\chi \|\lap s_n\|_2^2=\|s_n\|_S^2\leq \|f_n\|_S^2=\|f_n\|_2^2+\chi\|p_{n-1}\|_2^2$$ and H\"older's inequality. 

The upper bound of the second estimate is established in a similar manner using the last inequality. The lower bound is proved as follows.
\begin{align}
    \|s_n\|_S^2&=\|s_n\|_2^2+\chi\|\lap s_n\|_2^2\\
    &\geq \|p_n\|_2^2+\chi \|p_{n-1}-a_{n-1}\lap s_{n-1} -\tilde{b}_{n-1}\lap s_{n-2}\|_2^2\\
    & \geq  \|p_n\|_2^2+\chi \|p_{n-1}\|_{2}^2 \geq \qty(1+\chi\|G\|_2^{-2}) \|p_n\|_2^2,
\end{align}
 where we used the fact that $\|p_{n}\|_2\leq \|G\|_2 \|p_{n-1}\|_2$ which is proved in \cite[Theorem 3.4]{OST}.

Next, we see that
\begin{align}
|a_n|&=\frac{|\ip{f_{n+1}}{s_n}_S|}{\|s_n\|_S^2}=\frac{|\ip{f_{n+1}}{s_{n}}_2+\chi \ip{p_n}{\lap s_n}_2|}{\|s_n\|_S^2}\\
&= \frac{|\ip{f_{n+1}}{s_{n}}_2|}{\|s_n\|_S^2} \leq \frac{\|f_{n+1}\|_{2}\|s_{n}\|_{2}}{\|s_{n}\|_S^2}\\
& \leq \frac{\qty(\|G\|_{2}\|p_n\|_2)\qty(\|G\|_{2}\|p_{n-1}\|_{2})}{\chi \|p_{n-1}\|_{2}^2} \leq \chi^{-1}\|G\|_{2}^3.
\end{align}
At the same time, we have 
$$|a_n|\leq \frac{\|f_{n+1}\|_{2}\|s_{n}\|_{2}}{\|s_{n}\|_S^2} \leq \frac{\|f_{n+1}\|_{2}}{\|s_{n}\|_S} \leq \frac{\|G\|_{2}\|p_n\|_2}{\|p_{n}\|_2}=\|G\|_{2}.$$

Finally, 
$$\tilde{b}_{n}=\frac{\ip{f_{n+1}}{s_{n-1}}_S}{\|s_{n-1}\|_S^2}=\frac{\|p_{n}\|_{2}^2}{\|s_{n-1}\|_S^2}=\frac{\|p_{n}\|_{2}^2}{\|s_{n-1}\|_2^2+\chi\|\lap s_{n-1}\|_2^2}< \frac{\|p_{n}\|_{2}^2}{\|p_{n-1}\|_2^2+\chi\| p_{n-1}\|_2^2},$$ 
where we have used the fact that $\|\lap s_{n}\|_2> \|p_n\|_2.$ 
We can thus obtain the estimates for $\tilde b_n$ by equations \eqref{SobolevL2H1Estimate}. The rest of the asymptotics easily follows.
\end{proof}

\begin{remark}\label{rem:decayabn}
We can also prove that 
 $|a_n|\leq \chi^{-1}\|G\|_2^2\frac{\|p_n\|_2}{\|p_{n-1}\|_2}.$
 
     Intuitively it seems that $\lim\limits_{n\to \infty} \frac{\|p_n\|_2}{\|p_{n-1}\|_2}= 0$ as  $\lim\limits_{n\to \infty}\|p_n\|_2=0$ at a rate faster than exponential. However, we have not been able to prove this.
 
 Using the upper bounds in equations \eqref{SobolevL2H1Estimate} and \eqref{anbbounds} we can see that both $\|s_n\|_2$ and $\|s_n\|_S$  decay quickly, due to the decay of $\|p_n\|_2$ norms of the Legendre OPs which was proved  in \cite{OST}.
 
\end{remark}

When $\chi> 0$, we have the following estimate  for $\|s_n\|_\infty$.
 
\begin{corollary}\label{Corollary: laplacian-L2 and Linfinity estimates}
Under above conditions, there exists a positive constant $C>0$ such that for $n \ge 1$, we have 
\begin{equation}\label{L2EstimateofSobolevLaplacian}
   \begin{cases}
  \|\lap s_n\|_{2}^2\le \qty(1+\chi^{-1}\|G\|_{2}^2)\|p_{n-1}\|_{2}^2\\
\|s_n\|_{\infty}\le C\qty(1+\chi^{-\frac12})\|p_{n-1}\|_{2}.
\end{cases}
\end{equation}
\end{corollary}

\begin{proof}
The first estimate in~\eqref{L2EstimateofSobolevLaplacian} follows directly  from the second estimate of \eqref{SobolevL2H1Estimate}. For the second estimate, we   note that for any $u\in \operatorname{dom}\lap$, there is a constant $C>0$ such that $\|u\|_{\infty}\le C\qty(\|u\|_{2}+\|\lap u\|_{2})$ \cite[4.16, Lemma 4.6]{SU}. 
\end{proof}

By using the estimates in Proposition \ref{Proposition: L2, H1, coefficients estimates} and the recurrence relations in Theorem~\ref{th:k23recurrence}, we have the following asymptotic properties for $\qty{s_n\qty(:,\chi)}$ when $\chi$ tends to $\infty$.

\begin{corollary}\label{Corollary: Convergence result for lambda} There exists a positive constant $C>0$ such that for all $n\ge3$  we have 
\begin{equation*}
   \begin{cases}   \|s_n\qty(\cdot, \chi)-f_n\|_{2}\le 2\chi^{-1}\|G\|_{2}^5\|p_{n-3}\|_{2},\\
      \|s_n\qty(\cdot, \chi)-f_n\|_{\infty}\le C\qty(\chi^{-1}+\chi^{-\frac 32})\|p_{n-3}\|_{2},\\
      \|s_n\qty(\cdot, \chi) -f_n\|_S \leq \sqrt 2\chi^{-1}\sqrt{\|G\|_2^2+\chi}\,\|G\|_2^4\|p_{n-3}\|_2.
      \end{cases}
\end{equation*}
  
 Consequently,
 \begin{equation}\label{Asymptotic extimate of S_n(lambda) to fn}
\begin{cases}
\lim\limits_{\chi \to \infty} s_n\qty(x;\chi)=f_n\qty(x),\\
\lim\limits_{\chi \to \infty} \chi\qty(s_n\qty(\cdot, \chi)-f_n)=-\frac{\inner{f_n}{f_{n-1}}_{2}}{\|p_{n-2}\|_{2}^2}f_{n-1}-\frac{\|p_{n-1}\|_{2}^2}{\|p_{n-3}\|_{2}^2}f_{n-2}, \\
\lim\limits_{\chi \to \infty} \lap s_n \qty(x;\chi)= p_{n-1}(x),
\end{cases}
\end{equation}
where the limits are uniform in both $x$ and $n$.
\end{corollary}

\begin{proof}
The first and third estimates come from the recurrence relation \eqref{3-term recurrence} and the estimates for $a_n$, $\tilde b_n$ and $s_n$ in Proposition \ref{Proposition: L2, H1, coefficients estimates}. The second estimate is derived similarly, except that we also need Corollary \ref{Corollary: laplacian-L2 and Linfinity estimates} for $L^\infty$  estimate. The uniform convergence of $s_n\qty(x, \chi)$ to $f_n\qty(x)$ is a direct consequence of this.

Finally, observe that $ \chi\qty(s_n\qty(\cdot, \chi) - f_n)=-\chi a_{n-1}s_{n-1}-\chi \tilde{b}_{n-1}s_{n-2}.$ The result follows again from Proposition \ref{Proposition: L2, H1, coefficients estimates}.

\end{proof}

\begin{remark}\label{remark: Convergence result not true for n<3}
Corollary \ref{Corollary: Convergence result for lambda} is not true for $n<3$. For example, $s_0=P_0:=P_{0,k}$, and $s_1=P_1-\frac{\inner{P_1}{P_0}_{2}}{ \|P_0\|_{2}^2}P_0=p_1$. By using \eqref{3-term recurrence}, $s_2\qty(\cdot, \chi)$ converges to $f_2-\frac{\|p_1\|_{2}^2}{\|p_0\|_{2}^2}p_0$ uniformly as $\chi\rightarrow\infty$.

We also observe that $\qty{f_n\qty(x)=\lim\limits_{\chi \to \infty}s_n\qty(x, \chi)}$ is not an OP family. Indeed, $$ \ip{f_n}{f_m}_S=\begin{cases}
 \ip{f_n}{f_m}_2, \, \text{when}\, n\neq m,\\
 \|f_n\|_2^2+\chi\|p_{n-1}\|_2^2=\|f_n\|_2^2+\chi d_{n-1}^{-2}, \, \text{when}\, n=m
 \end{cases}$$
 However, for $|n-m|\geq 3,$ 
 $$\ip{f_{n}}{f_{m}}_2=\ip{p_n+b_{n-1}p_{n-1}+c_{n-1}p_{n-2}}{p_m+b_{m-1}p_{m-1}+c_{m-1}p_{m-2}}_2=0.$$
 Thus, the set of polynomials $\qty{f_{n}}_{n=0}^{\infty}$ where $f_0=0$ is ``almost'' orthogonal with respect to both the standard inner product as well as with the Sobolev inner product.
Similarly, $\qty{-\frac{\inner{f_n}{f_{n-1}}_{2}}{\|p_{n-2}\|_{2}^2}f_{n-1}-\frac{\|p_{n-1}\|_{2}^2}{\|p_{n-3}\|_{2}^2}f_{n-2}=\lim\limits_{\chi \to \infty}\chi\qty(s_n\qty(\cdot, \chi)-f_n)}$ is also ``almost'' orthogonal with respect to both the standard inner product as well as with the Sobolev inner product.
\end{remark}

 We can use these remarks to construct a related family of orthogonal polynomials $\qty{\tilde{f}_n\qty(\cdot, \chi)}_{n\geq 0}$ with respect to $\ip{\cdot}{\cdot}_S$ as follows
 \begin{equation}\label{asso-poly}
 \begin{cases}
 \tilde{f}_0=f_0=0,\\
 \tilde{f}_1=f_1,\\
 \tilde{f}_n\qty(x, \chi) =f_n\qty(x) +t_{n}\qty(\chi)\tilde{f}_{n-1}\qty(x, \chi)+ u_n\qty(\chi) \tilde{f}_{n-2}\qty(x, \chi),\quad n\geq 2,
 \end{cases}
\end{equation}

where the sequences $\qty{t_n\qty(\chi)}_{n\geq 2}$ and $\qty{u_n\qty(\chi)}_{n\geq 2}$ are chosen so that $\ip{\tilde{f}_n\qty(\cdot, \chi)}{\tilde{f}_m\qty(\cdot, \chi)}_S=0$ for $n\neq m$.

We note that a similar argument was used in \cite{Coh75} to construct variations of the classical Sobolev-Legendre OPs. 

\begin{proposition}\label{prop:orthoasso}
There exist coefficients $\qty{t_n\qty(\chi)}_{n\geq 2}$ and $\qty{u_n\qty(\chi)}_{n\geq 2}$ such that $\qty{\tilde{f}_n\qty(\cdot, \chi)}_{n\geq 0}$ given by~\eqref{asso-poly} is an orthogonal set of polynomials in the Sobolev inner product space.
\end{proposition}

\begin{proof}
For $n=2$, we only need to find $t_2\qty(\chi)$ such that $$\ip{\tilde{f}_2\qty(\cdot, \chi)}{\tilde{f}_1\qty(\cdot, \chi)}_S= \ip{f_2+ t_2\qty(\chi)f_1}{f_1}_S= \ip{f_2}{f_1}_2+t_2\qty(\chi)\|f_1\|_S^2=0.$$ Note that $t_2\qty(\chi)\neq 0$ since we can check that $\ip{f_2}{f_1}_2\neq 0$. Using Theorems\cite[Theorem 3.2]{OST} and \cite[Theorem 3.4]{OST} we see that
$$\ip{f_2}{f_1}_2= \ip{p_2+b_1p_1+c_1p_0}{p_1+b_0p_0}_2=b_1d_1^{-2}+c_1b_0d_0^{-2}=b_1d_1^{-2}+b_0d_1^{-2}=d_1^{-2}\qty(b_0+b_1)<0.$$ In fact, $t_2\qty(\chi)=-\tfrac{d_1^{-2}\qty(b_0+b_1)}{d_1^{-2}+b_0^2d_0^{-2}+\chi d_0^{-2}}>0.$

For $n=3$, we must find $t_3\qty(\chi), u_3\qty(\chi)$ such that $\ip{\tilde{f}_3\qty(\cdot, \chi)}{\tilde{f}_2\qty(\cdot, \chi)}_S=\ip{\tilde{f}_3\qty(\cdot, \chi)}{\tilde{f}_1\qty(\cdot, \chi)}_S=0.$ To see that this is always possible, we proceed as follows.
$$ \ip{\tilde{f}_3\qty(\cdot, \chi)}{\tilde{f}_1\qty(\cdot, \chi)}_S=\ip{f_3+ t_3\qty(\chi)\tilde{f}_2\qty(\cdot, \chi) + u_3\qty(\chi)\tilde{f}_1\qty(\cdot, \chi)}{\tilde{f}_1\qty(\cdot, \chi)}_S= \ip{f_3}{f_1}_S+   u_3\qty(\chi)\ip{f_1}{f_1}_S$$
We note that $\ip{f_3}{f_1}_S=\ip{f_3}{f_1}_2=c_2d_1^{-2}=d_2^{-2}$ from which we get $u_3\qty(\chi)=-\tfrac{d_{2}^{-2}}{\|f_1\|_2^2+\chi d_0^{-2}}< 0$.

Similarly, 
$$ \ip{\tilde{f}_3\qty(\cdot, \chi)}{\tilde{f}_2\qty(\cdot, \chi)}_S=\ip{f_3}{\tilde{f}_2\qty(\cdot, \chi)}_S+t_3\qty(\chi)\|\tilde{f}_2\qty(\cdot, \chi)\|_S^2=0.$$
However, $\ip{f_3}{\tilde{f}_2}_S=\ip{f_3}{f_2}_S=\ip{f_3}{f_2}_2=b_2d_2^{-2}+b_1d_1^{-2}.$ Hence,
$t_3\qty(\chi)=-\tfrac{b_2d_2^{-2}+b_1d_1^{-2}}{\|\tilde{f}_2\qty(\cdot, \chi)\|_S^2}.$

The rest of the proof proceed by induction. By construction $\ip{\tilde{f}_{n}}{\tilde{f}_{n+k}}_S=0$ for any $n\geq 0$ and $k=1,2,3.$ For any $n, m$ such that $|n-m|>3$ we see that $\ip{\tilde{f}_n}{\tilde{f}_m}_S=0$ from the fact that $\ip{f_n}{f_m}_S=0$ for all such indices. 
\end{proof}

\subsection{Sobolev Orthogonal Polynomials with respect to \texorpdfstring{$k=1$}{k=1}}\label{sec: SOP wrt k=1}
In this section, we consider the Sobolev inner product 
$$\ip{f}{g}_S=\int_{SG}f\qty(x)g\qty(x)\, \dd{\mu}\qty(x)+\chi \int_{SG} \Delta f\qty(x)\Delta g\qty(x)\, \dd{\mu}\qty(x)$$ defined earlier in Section~\ref{sec:sopsg}, and study the corresponding Sobolev OPs constructed  from the family of monomials with $k=1$. In this case, we will show that the resulting Sobolev OPs satisfy a four-term recurrence relation instead of a three-term. This will lead to  slight differences in the estimates and properties of these polynomials.
We will abuse the notations and still denote by $\qty{s_{n}}_{n=0}^{\infty}$, the Sobolev OPs for the family $k=1$.

We first state and prove a version of Lemma~\ref{Lemma: Family 2 and 3 is self contained on Green's, Family 1 isn't} that holds only under some restrictions.

\begin{lemma}\label{Lemma: Family 1 is not self contained}
Fix $k=1$. Let $C$ be a polynomial in $\operatorname{span}\qty{P_{n,1}}_{n=0}^{\infty}$ with $\operatorname{deg}\qty(C) = J$. Define the function $f$ on $SG$ by $f\qty(x):= -\int_{SG} G\qty(x,y)C\qty(y)dy$. Then, $f$ is also in $\operatorname{span}\qty{P_{n,1}}_{n=0}^{\infty}$ with $\operatorname{deg}\qty(f) = J+1$ if and only if $\partial_n f\qty(q_0)=0$.
\end{lemma}

\begin{proof} The proof is the same as that of Lemma~\ref{Lemma: Family 2 and 3 is self contained on Green's, Family 1 isn't} except that when $k=1$, by symmetry, coefficient of $P_{0,3}$ is $0$ but $f$ may include a term from $P_{0,2}$ which can be only eliminated when $\partial_n f\qty(q_0) = 0$. Conversely, if the coefficient on $P_{0,2}$ was 0, then $\partial_n f\qty(q_0) = 0$ because $\partial_n P_{i,1}\qty(q_0)=0$ for any $i$. 
\end{proof}

Results similar to the ones for $k=2,3$ which were proved in Section~\ref{sec3.2} are also valid in some sense for $k=1$. However, we could establish these results only if we assume that the following conjecture is true. The statement uses the  the notations in Lemma~\ref{Lemma: Expansions for f}.
\begin{conjecture}\label{conjecture:normal} 
For any integer $t\geq 0$, let $$f_{t+1,1} := -\int_{SG}G\qty(x,y)p_{t,1}\qty(y)dy.$$  
We have
\begin{align}
    \partial_n f_{t+1,1}\qty(q_0)\neq 0. \label{equation: normal derivative conjecture}
\end{align}
\end{conjecture} 

In light of Lemma~\ref{Lemma: Family 1 is not self contained}, Conjecture~\ref{conjecture:normal} is equivalent to the fact that $f_{t+1,1}$ does not belong to the $k=1$ family. This is in sharp contrast to the situations for $k=2$ or $k=3$. While, we have not been able to establish the conjecture,  we do have strong numerical evidences that it is true, and for the rest of this section we shall assume so. 

The first result gives some norm estimates for $s_n$ which should be compared to Theorem~\ref{thm:normestimate}

\begin{proposition}\label{k1 L2 and S2 estimates} For $k=1$, the Sobolev OPs satisfy the following additional estimates for $n\geq 1$.
\begin{equation}\label{SobolevH1estimatesk=1}
\begin{cases}
\|p_n\|_{2}^2+\chi\|p_{n-1}\|_{2}^2\leq \|s_n\|^2_S\leq 2\|G\|_{2}^2\|p_{n-1}\|^2_{2}+\partial_nf_n\qty(q_0)^2+\chi\|p_{n-1}\|_{2}^2,\\
\|p_n\|_2\leq \|s_{n}\|_{2}\leq \|G\|_{2}\|p_{n-1}\|_{2}+\qty|\partial_nf_n\qty(q_0)|.
\end{cases}
\end{equation}
\end{proposition}

\begin{proof}
Let $g:=f_n-\partial_nf_n\qty(q_0)P_{0,2}$, then $\partial_n g\qty(q_0)=0$, hence by Lemma~\ref{Lemma: Family 2 and 3 is self contained on Green's, Family 1 isn't}, it is a polynomial spanned by $\qty{P_{0,1}}$. It follows that
\begin{align}
&2\|G\|_{2}^2\|p_{n-1}\|_{2}^2+\partial_nf_n\qty(q_0)^2+\chi\|p_{n-1}\|_{2}^2\\
&\ge\qty(\|f_n\|_{2}+\qty|\partial_nf_n\qty(q_0)|\,\|P_{0,2}\|_{2})^2+\chi\|p_{n-1}\|_{2}^2 \\
&\ge\|g\|_S^2\ge\|s_n\|^2_S\ge\|s_n\|_{2}^2+\chi\|p_{n-1}\|_{2}^2\ge\|p_n\|_{2}^2+\chi\|p_{n-1}\|_{2}^2    
\end{align}\end{proof}

The following is the analog of the three-term recursion formula in the context of Sobolev OPs starting from the monomials in the the $k=1$ family. Observe that it is different in nature, as the right hand side involves two terms. 

\begin{theorem}\label{Recurrence Relation with one assumption $(k=1)$} Let $\qty{s_n}$ be the monic Sobolev orthogonal polynomials and $\qty{p_n}$ the monic Legendre polynomials generated from the $k = 1$ family of monomials. Let $s_{-1}:=0$, $f_{n+2}\qty(x) = -\int_{SG}G\qty(x,y)p_{n+1}\qty(y)dy$ and suppose that $\partial_n f_{n+2}\qty(q_0)\neq 0$.  Then the following statements hold.

\begin{enumerate}
    \item[(1)] For each integer $n\geq-1$, the Sobolev OPs $\qty{s_n}$ satisfy the following recurrence relation:
    \begin{equation}
    s_{n+3}+a_ns_{n+2} + b_ns_{n+1}+c_ns_n = f_{n+3}+d_nf_{n+2},  \label{recurrencek=1}
    \end{equation}
    where the coefficients are given by
    \medskip
    \begin{equation} \label{a_n and b_n in k=1 recurrence}
    \begin{cases}
    a_n=\frac{\inner{f_{n+3}+d_nf_{n+2}}{s_{n+2}}_S}{\|s_{n+2}\|_S^2},\; \; 
    b_n= \frac{\inner{f_{n+3}+d_nf_{n+2}}{s_{n+1}}_S}{\|s_{n+1}\|_S^2},\\
    d_n=-\frac{\partial_n f_{n+3}\qty(q_0)}{\partial_nf_{n+2}\qty(q_0)},\; \; 
    c_n=-d_n\frac{\|p_{n+1}\|_{2}^2}{\|s_n\|_S^2}.
    \end{cases}
  \end{equation}
   \item[(2)] For each fixed $n\ge1$, 
 \begin{equation*}
     \begin{cases}
     |b_n|=O\qty(\chi^{-1}), \quad |c_n|=O\qty(\chi^{-1}),\\
     \lim_{\chi \to \infty} a_n=-d_n, \quad \lim_{\chi \to \infty}\chi c_n= -d_n\frac{\|p_{n+1}\|_{2}^2}{\|p_{n-1}\|_{2}^2}.
      \end{cases}
\end{equation*}

    

\end{enumerate}

\end{theorem}

\begin{proof} 
Because of the assumption that $\partial_n f_{n+2}\qty(q_0)\neq 0$, we need to choose $d_n$ such that $\partial_n f_{n+3}\qty(q_0)+ d_n\partial_n f_{n+2}\qty(q_0)= 0$ to ensure that the polynomial $f_{n+3}+d_nf_{n+2}$ remains in the $k=1$ family according to Lemma~\ref{Lemma: Family 1 is not self contained}. Thus we see that $d_n=-\frac{\partial_n f_{n+3}\qty(q_0)}{\partial_nf_{n+2}\qty(q_0)}.$

It follows that $f_{n+3}+d_nf_{n+2}$ can be written as a finite linear combination of monomials in  $\qty{P_{n,1}}_{n=0}^{\infty}$, and it vanishes on the boundary and has zero normal derivatives. For any $t<n$, let $g$
 be a polynomial in the finite span of $\qty{P_{n,1}}_{n=0}^{\infty}$ such that $\lap g=s_t$. Then $\inner
{f_{n+3}+d_nf_{n+2}}{s_t}_S=\int_{SG}\qty(f_{n+3}+d_nf_{n+2})\lap g\, \dd{\mu}=-\int_{SG} \qty(p_{n+2}+d_np_{n+1})g\,
\dd{\mu}=0$. 

Next, take $h$ to be a monic polynomial in the finite span of $\qty{P_{n,1}}_{n=0}^{\infty}$ such that $\lap
h=s_{n}$. Then $$ c_n\|s_n\|_S^2=\inner{f_{n+3}+d_nf_{n+2}}{s_{n}}_S=-\int
\qty(p_{n+2}+d_np_{n+1}) h\, \dd{\mu}= -d_n \|p_{n+1}\|_{2}^2.$$  
The expressions for $a_n$ and $b_n$ are trivially derived. 

As for the estimates, one sees that $a_n=\frac{\int\qty(f_{n+3}+d_n f_{n+2})s_{n+2}\,\dd{\mu}}{\|s_{n+2}\|^2_S}-d_n\chi\frac{\int p^2_{n+1}\,\dd{\mu}}{\|s_{n+2}\|^2_S}$, the first term is $O\qty(\chi^{-1})$, while the second term converges to $-d_n$ as $\chi$ goes to $\infty$ by Proposition \ref{k1 L2 and S2 estimates}. 

The other
arguments are just the same as in the proof of Theorem \ref{th:k23recurrence}.
\end{proof}


The next result is an analog of Corollary~\ref{Corollary: Convergence result for lambda} in the case of $k=1$.

\begin{corollary} Assume $k=1$ and Conjecture \ref{conjecture:normal} is true. Then there exists a sequence of monic polynomials $\qty{g_n}_{n=0}^{\infty}$ independent of $\chi$ such that for any $n\ge0$, $deg\, g_n=n$, $\lim_{\chi \to \infty}s_n(x; \chi)=g_n(x)$ where the convergence is uniform in $x$ and $n$.   Furthermore, $g_{n+3}+d_ng_{n+2}=f_{n+3}+d_nf_{n+2}$ for any $n\ge 1$, where $d_n$ is given in Theorem~\ref{Recurrence Relation with one assumption $(k=1)$}. For the basic cases, $g_0=p_0$, $g_1=p_1$, $g_{2}+d_{-1}g_{1}=f_{2}+d_{-1}f_{1}-\frac{\inner{f_2+d_{-1}f_1}{g_0}_{2}}{{\|g_0\|_{2}^2}}g_{0}$, and 
$g_{3}+d_{0}g_{2}=f_{3}+d_{0}f_{2}-\frac{\inner{f_3+d_{0}f_2}{g_0}_{2}}{{\|g_0\|_{2}^2}}g_{0}$. Moreover, for any $\alpha<1$, $n\ge0$, $\lim\limits_{\chi \rightarrow\infty}\chi^\alpha\qty(s_n\qty(\chi)-g_n)=0$ uniformly in $x$.
\end{corollary}

\begin{proof} The proof is similar to that of Corollary~\ref{Corollary: Convergence result for lambda}.


\end{proof}

\subsection{Orthogonal polynomials with respect to the higher order Sobolev-type  inner products}\label{subsec:higherordersop}
We now investigate Sobolev OPs with respect to  higher order Sobolev  inner products. Because most of the results are similar to the ones established in the last three sections, we shall omit most of the proofs.

\begin{definition} \label{def:innergeneral}
For any integer  $m\geq1$, the Sobolev-$m$ inner product $\inner{\cdot}{\cdot}_{S^m}$ is defined as
\begin{equation} \label{eq:sobk}
    \inner{f}{g}_{S^m} = \sum\limits_{\ell = 0}^m \chi_\ell\int_{SG}\lap^\ell f\lap^\ell g\,\dd{\mu}=\int_{SG} fg\, \dd{\mu} + \sum\limits_{\ell = 1}^m \chi_\ell \int_{SG}\lap^\ell f\lap^\ell g\,\dd{\mu},
    \end{equation}
 where $\chi_\ell$ are all non-negative constants, $\chi_0:=1$. 
\end{definition}

The next result collects formulas for computing some specific inner products of the monomials needed to represent the orthogonal polynomials as linear combinations of these monomials  \cite[Lemma 2.1]{OST}.
\begin{lemma}\label{lemma:innerproducthigher}
Suppose $m\in \N$, $\delta_0:=1$, $\chi_1,\dots, \chi_m\ge 0$  in~\eqref{eq:sobk}. 
 Then the following statements hold. 
 
 \begin{equation*}
 \begin{cases}
 \inner{P_{j,1}}{P_{k,1}}_{S^{m}} = 2\sum\limits_{r=0}^{m}\chi_{r}\sum\limits_{l=j-m_{*}}^{j}\qty(\alpha_{j-l-r}\eta_{k+l+1-r}-\alpha_{k+l+1-r}\eta_{j-l-r}),\\
 \inner{P_{j,2}}{P_{k,2}}_{S^{m}} = -2\sum\limits_{r=0}^{m}\chi_{r}\sum\limits_{l=j-m_{*}}^{j}\qty(\beta_{j-l-r}\alpha_{k+l+1-r}-\beta_{k+l+1-r}\alpha_{j-l-r}'),\\
 \inner{P_{j,3}}{P_{k,3}}_{S^{m}} = 18\sum\limits_{r=0}^{m}\chi_{r}\sum\limits_{l=j-m_{*}}^{j}\qty(\alpha_{j-l+1-r}\eta_{k+l+2-r}-\alpha_{k+l+2-r}\eta_{j-l+1-r}),\\
 \inner{P_{j,1}}{P_{k,2}}_{S^{m}} = -2\sum\limits_{r=0}^{m}\chi_{r}\sum\limits_{l=j-m_{*}}^{j}\qty(\alpha_{j-l-r}\alpha_{k+l+1-r}+\beta_{k+l+1-r}\eta_{j-l-r}),\\
    \inner{P_{j,1}}{P_{k,3}}_{S^{m}}=\inner{P_{j,2}}{P_{k,3}}_{S}=0,\\
    \inner{P_{j,3}^{\qty(n)}}{P_{k,3}^{\qty(n)}}_{S^{m}}=\inner{P_{j,3}^{\qty(0)}}{P_{k,3}^{\qty(0)}}_{S^{m}},\\
    \inner{P_{j,3}^{\qty(n)}}{P_{k,3}^{\qty(n')}}_{S^{m}}=-\frac{1}{2}\inner{P_{j,3}^{\qty(0)}}{P_{k,3}^{\qty(0)}}_{S^{m}},
 \end{cases}    
 \end{equation*}
 where $\alpha_i'=\frac12$ if $i=0$; otherwise $\alpha_i'=\alpha_i$. $\alpha_i=\beta_i=\eta_i=0$ if $i<0$. $m_*:=\min\qty{j, k}$.



\end{lemma}
\begin{proof}
We use  Lemma 2.1 in \cite[Lemma 2.1]{OST} along with the following observation
\begin{align*}
    \inner{P_{j,i}}{P_{ki'}}_{S^{m}} &= \int_{\sg} P_{j,i}P_{ki'} d \mu  + \sum\limits_{r=1}^{m}\chi_{r}\int_{\sg} \lap^{r} P_{j,i}\lap^{r} P_{ki'} d \mu\\
    &= \int_{\sg} P_{j,i}P_{ki'} d \mu  +\sum\limits_{r=1}^{m} \chi_{r}\int_{\sg} P_{j-r, i} P_{\qty(k-r)i'} d \mu 
    = \sum\limits_{r=0}^{m} \chi_{r}\int_{\sg} P_{j-r, i} P_{\qty(k-r)i'} d \mu.
\end{align*}

\end{proof}

We denote by $W^{m,2}$ the Hilbert space given by this inner product. Fixing $m\geq2$, and using this inner product for fixed $k=1,2$ or $3$, we apply the Gram-Schmidt algorithm to the sequence of polynomials $\qty{P_{n,k}^{\qty(0)}}_{n=0}^{\infty}$ to get the Sobolev orthogonal polynomials (with respect to $q_0$). By an abuse of notation, we still call the resulting functions, the Sobolev OPs and denote them by $\qty{\tilde{s}_{n, k}\qty(x;X)}_{n=0}^{\infty}$, where $X=\qty{\chi_{\ell}}_{\ell=1}^m$. When there is no confusion about $k$ and $X$ we will simply write $\qty{\tilde{s}_{n}}_{n=0}^{\infty}$. The corresponding orthonormal polynomials will be denoted $\qty{\tilde{S}_{n, k}\qty(x;X)}_{n\geq 0}$ or $\qty{\tilde{S}_{n}}_{n\geq 0}$ when there is no confusion.

The first result we prove is the following generalization of Theorem~\ref{th:k23recurrence} to the higher order Sobolev  inner product for a fixed  $k=2$ or $3$. For convenience, we denote by $\mathcal{G}^m$ the $m$ fold composition of the Green operator, where $m\geq 2$ is an integer.

\begin{theorem}\label{th:gen_rec}
Fix an integer $m\geq 2$ and assume that $k=2$ or $3$. Let 
$$\mathcal{F}_{m+j}\qty(x)=\mathcal{G}^{m}p_{j}\qty(x),$$ where $\{p_{j}\}_{j\geq 0}$ denotes the corresponding Legendre polynomials. 
For the higher order Sobolev inner product~\eqref{eq:sobk}, we have the following generalized recursion relation for $n\ge -1$
\begin{equation}\label{generalrecur}
\begin{cases}
\mathcal{F}_{n+m+1}=\tilde{s}_{n+m+1} + \sum\limits_{\ell = 0}^{2m-1}a_{n, \ell}\tilde{s}_{n+m-\ell}, \\
a_{n, \ell} = \frac{\inner{\mathcal{F}_{n+m+1}}{\tilde{s}_{n+m-\ell}}_{S^m}}{\inner{\tilde{s}_{n+m-\ell}}{\tilde{s}_{n+m-\ell}}_{S^m}},
\end{cases}
\end{equation}
and   $\tilde{s}_j:=0$ if $j<0$. 
\end{theorem}

\begin{proof}
Let $g_n = \mathcal{F}_{n+m+1}-\tilde{s}_{n+m+1} - \sum\limits_{\ell = 0}^{2m-1}a_{n, \ell}\tilde{s}_{n+m-\ell}$. We know that $g_n$ has degree $\leq n+m$. Consider $\inner{g_n}{\tilde{s}_t}_{S^m}$ for $t < n+m$. For $n-m+1 \leq t \leq n+m$, it follows from the definition that $\inner{g_n}{\tilde{s}_t}_{S^m}=0$. For $0\le t < n-m+1$, we have
\begin{align}
\inner{g_n}{\tilde{s}_t}_{S^m} = -\inner{\mathcal{F}_{n+m+1}}{\tilde{s}_t}_{S^m} = \sum\limits_{\ell = 0}^m \chi_\ell\int_{SG}\lap^\ell \mathcal{F}_{n+m+1}\lap^\ell \tilde{s}_t\,\dd{\mu}\\
=\sum\limits_{\ell = 0}^m \chi_\ell\int_{SG}\mathcal{G}^{m-\ell}p_{n+1}\lap^\ell \tilde{s}_t=\sum\limits_{\ell = 0}^m \chi_\ell\int_{SG}p_{n+1}\mathcal{G}^{m-\ell}\qty(\lap^\ell \tilde{s}_t)=0.
\end{align}
where the last equality follows from Lemma \ref{Lemma: Family 2 and 3 is self contained on Green's, Family 1 isn't}. Thus, we have shown that $g_n = 0$.
\end{proof}

Similarly to the asymptotics analysis of the Sobolev OPs when $m=1$ done in Section~\ref{sec3.2}, we now state and give short proofs of analogous results for higher order Sobolev OPs.

\begin{corollary} \label{cor:k=2,3 bounds} Let $m\geq 2$ and fix $k=2$ or $3$. Suppose that $\{\chi_\ell\}_{\ell=0}^m$  is such that $\chi_\ell \leq \chi_m$ for each $0\leq \ell \leq m$. Then there exist positive constants $C_1=C_1\qty(n,\mu, m)$, $C_2=C_2\qty(n,\mu,m)$ such that for any $n\ge 0$, 
\begin{equation}\label{estimatesnormsm}
C_1 \leq  \|\tilde{s}_n\|^2_{S^{m}} \leq C_2,\quad \text{and}\quad C_1+ \chi_m \|p_{n-m}\|_{2}^2 \leq  \|\tilde{s}_n\|_{S^m}^2 \leq C_2+\chi_m\|p_{n-m}\|_{2}^2.
\end{equation}
   Consequently, for any $n\ge 2m+1$, we have
$$\|\tilde{s}_n-\mathcal{F}_n\|_{2}\leq C\qty(n,M,m,\mu)\chi_m^{-1}, \quad \text{and}\quad 
   \lim\limits_{\chi_m\to \infty}\|\lap^i \tilde{s}_n-\mathcal{G}^{m-i}p_{n-m}\|_{\infty}=0.
$$
    for any $0\le i\le m$.
   \end{corollary}

   \begin{proof}
  The first two estimates follow from the fact that 
  $$C\qty(m,n,M,\mu)+\chi_m\|p_{n-m}\|_{2}^2\geq \|\mathcal{F}_n\|_{S^m}^2\geq \|\tilde{s}_n\|_{S^m}^2\geq \sum\limits_{\ell = 0}^{m-1} \chi_\ell\|\lap ^\ell \tilde{s}_n\|_2^2 +\chi_m\|p_{n-m}\|_{2}^2.$$ 
  However, $$ \sum\limits_{\ell = 0}^{m-1} \chi_\ell\|\lap ^\ell \tilde{s}_n\|_2^2 +\chi_m\|p_{n-m}\|_{2}^2\geq \|p_n\|_{2}^2+\chi_m\|p_{n-m}\|_{2}^2.$$

   Next, we estimate directly  the coefficients in recurrence. By using the Cauchy-Schwarz inequality for the inner product $\inner{f}{g}_{S^m} = \sum\limits_{\ell = 0}^{m-1} \chi_\ell\ip{\lap^\ell f}{\lap^\ell g}_2$ along with the fact that $$\ip{\lap^m \mathcal{F}_n}{\lap^m \tilde{s}_t}_2 =\ip{p_{n-m}}{\lap^m \tilde{s}_t}_2=0$$ for $t<n$, we have $|a_{n-1-m,\ell}|\le \chi_m^{-1}C\qty(m,n,\mu,M)$ for any $\ell$.
   
   Estimating the $L^2$ norm of $\mathcal{F}_n-\tilde{s}_n$ directly in the recurrence by the triangle inequality, and observing that all norms in a finite dimensional space are equivalent, completes the proof.
   \end{proof}

The next result is a generalization of Theorem \ref{th: sobolev ode}. For the sake of completeness we include its proof.

 \begin{theorem}\label{th: higher order sobolev ode}
Fix $k=2$ or $3$. Then for the higher order Sobolev inner product \eqref{eq:sobk}, the Sobolev orthogonal polynomials satisfy the following differential equation for each $n\geq m$: 
\begin{align}
\tilde s_n\qty(x)+\sum\limits_{l=1}^m \chi_l \Delta^{2l} \tilde s_n\qty(x)=\Delta^m p_{n+m}\qty(x)+\sum\limits_{l=1}^{2m}d_{n+m-l}^{2} \tilde \xi_n^{-2}a_{n+m-l-1,2m-l-1}\lap^m p_{n+m-l}, \label{th: Sobolev ODE} 
\end{align}
 where $\qty{p_n}_{n\geq 0}$ are the corresponding Legendre OPs, $\qty{a_{n,l}}$ is given as in Theorem \ref{th:gen_rec},  $d_n^{-2}=\|p_n\|_2^2$, and  $\tilde \xi_{n}^{-2}=\|\tilde s_n\|_{S^m}^2.$
\end{theorem}

\begin{proof}
Let $n\geq m$. Given any polynomial $h$ of degree at most $n-m-1$, let $g\qty(x)=\mathcal G^m h\qty(x)$. Since $g$ is a polynomial of degree at most $n-1$, we see that 
\begin{align}
    0&=\ip{\tilde s_n}{g}_{S^m}\\
    &=\int_{SG}\tilde s_n\qty(x)g\qty(x)\, \dd{\mu}\qty(x)+\sum\limits_{l=1}^m\chi_l \int_{SG} \lap^l \tilde s_n\qty(x) \lap^l g\qty(x)\, \dd{\mu}\qty(x)\\
    &=\int_{SG}\tilde s_n\qty(x)\mathcal G^{m} h\qty(x)\, \dd{\mu}\qty(x)+\sum\limits_{l=1}^m\chi_l \int_{SG} \lap^l \tilde s_n\qty(x) \mathcal G^{m-l} h\qty(x)\, \dd{\mu}\qty(x)\\
     &=\int_{SG}\mathcal G^{m}\tilde s_n\qty(x)h\qty(x)\, \dd{\mu}\qty(x)+\sum\limits_{l=1}^m\chi_l \int_{SG}  \mathcal G^{m-l}\qty(\lap^l \tilde s_n)\qty(x)  h\qty(x)\, \dd{\mu}\qty(x)\\
    &=\int_{SG}\qty(\mathcal G^{m}\tilde s_n\qty(x)+\sum\limits_{l=1}^m\chi_l\mathcal G^{m-l}\qty(\lap^l \tilde s_n)\qty(x)) \;h\qty(x)\, \dd{\mu}\qty(x).
    \end{align}
Note that $\mathcal G^{m}\tilde s_n\qty(x)+\sum\limits_{l=1}^m\chi_l\mathcal G^{m-l}\qty(\lap^l \tilde s_n)\qty(x)$ is a monic polynomial of degree $n+m$, thus we can write
$$\mathcal G^{m}\tilde s_n\qty(x)+\sum\limits_{l=1}^m\chi_l\mathcal G^{m-l}\qty(\lap^l \tilde s_n)\qty(x)=  p_{n+m}\qty(x)+\sum\limits_{i=1}^{2m}b_{n,i} \,p_{n+m-i}\qty(x),$$
where 
\begin{align}
    b_{n,i}&=d_{n+m-i}^{2}\ip{\mathcal G^{m}\tilde s_n+\sum\limits_{l=1}^m\chi_l\mathcal G^{m-l}\qty(\lap^l \tilde s_n)}{p_{n+m-i}}_2\\
    &=d_{n+m-i}^{2}\sum\limits_{l=0}^m\chi_l \ip{\lap^l \tilde s_n}{\mathcal G^{m-l}\qty(p_{n+m-i})}_2\\
    &=d_{n+m-i}^{2}\sum\limits_{l=0}^m\chi_l \ip{\lap^l \tilde s_n}{ \lap^l \mathcal F_{n+2m-i}}_2\\
    &=d_{n+m-i}^{2}\ip {\tilde s_n}{\mathcal F_{n+2m-i}}_{S^m}\\
    &=d_{n+m-i}^{2} \tilde \xi_n^{-2}a_{n+m-i-1,2m-i-1}.
\end{align}

Taking Laplacian $m$ times on both sides yields the result.
\end{proof}

\begin{remark}\label{remark: Other kinds of $H^1$ inner product}
Theorem \ref{th:gen_rec} may be established for the following more general inner product:  
\begin{align}\label{eq:point mass inner}
\inner fg_{S^m}= 
&\sum\limits_{\ell = 0}^m \chi_\ell\ip{\lap^\ell f}{\lap^\ell g}_2 +\sum\limits_{\ell=0}^{m-1}\chi_{\ell}'\,\mathscr{E}\qty(\lap^\ell f,\lap^\ell g)+\nonumber\\
&\sum\limits_{\ell=0}^{m-1
}\qty[\lap ^\ell f\qty(q_0)\,\lap ^\ell f\qty(q_1)\,\lap ^\ell f\qty(q_2)] M_\ell \qty[\lap ^\ell g\qty(q_0)\,\lap ^\ell g\qty(q_1)\,\lap ^\ell g\qty(q_2)]^T,
\end{align}
where 
$\chi_\ell$, $\chi'_\ell$ are non-negative, $M_\ell$ are positive semi-definite $3\times 3$ matrices. 


\end{remark}

\subsection{Numerical results}\label{subsec: numericsSLOPs}

We first consider the Sobolev-Legendre OPs constructed from the families of monomials corresponding to $k=2, 3$. In this case, the  recurrence relation~\eqref{3-term recurrence} allows us to recursively evaluate the anti-symmetric Sobolev orthogonal polynomials $s_n$, once $s_{n-1}, s_{n-2},$ and $f_n$ are known. This approach is used to generate and plot these polynomials on $\sg$ in Figure \ref{fig: k=2,3 H1 plots}. We note that the plots are approximations of the Sobolev OPs on the finite graph approximations $\Gamma_m$ of $\sg$. 
\begin{figure}[h]
    \centering
    \includegraphics[width=0.45\textwidth]{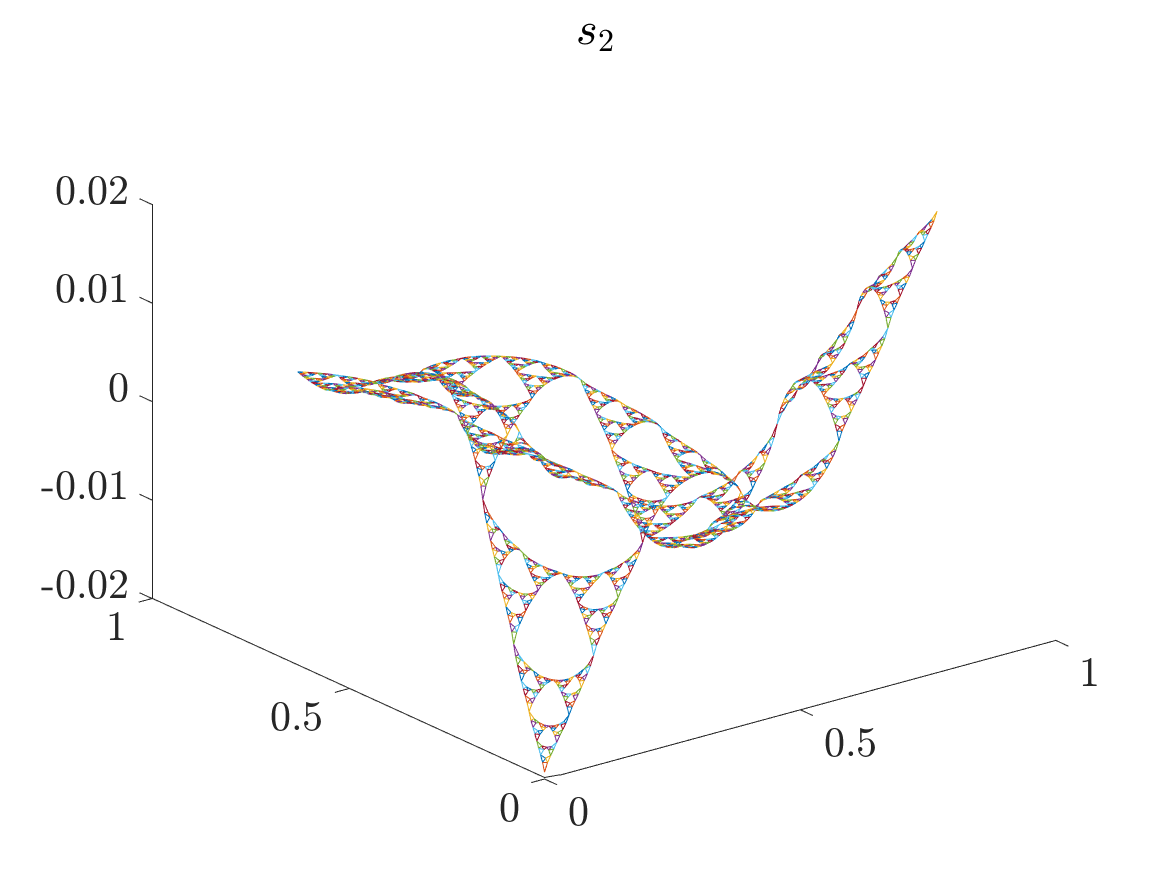}
    \includegraphics[width=0.45\textwidth]{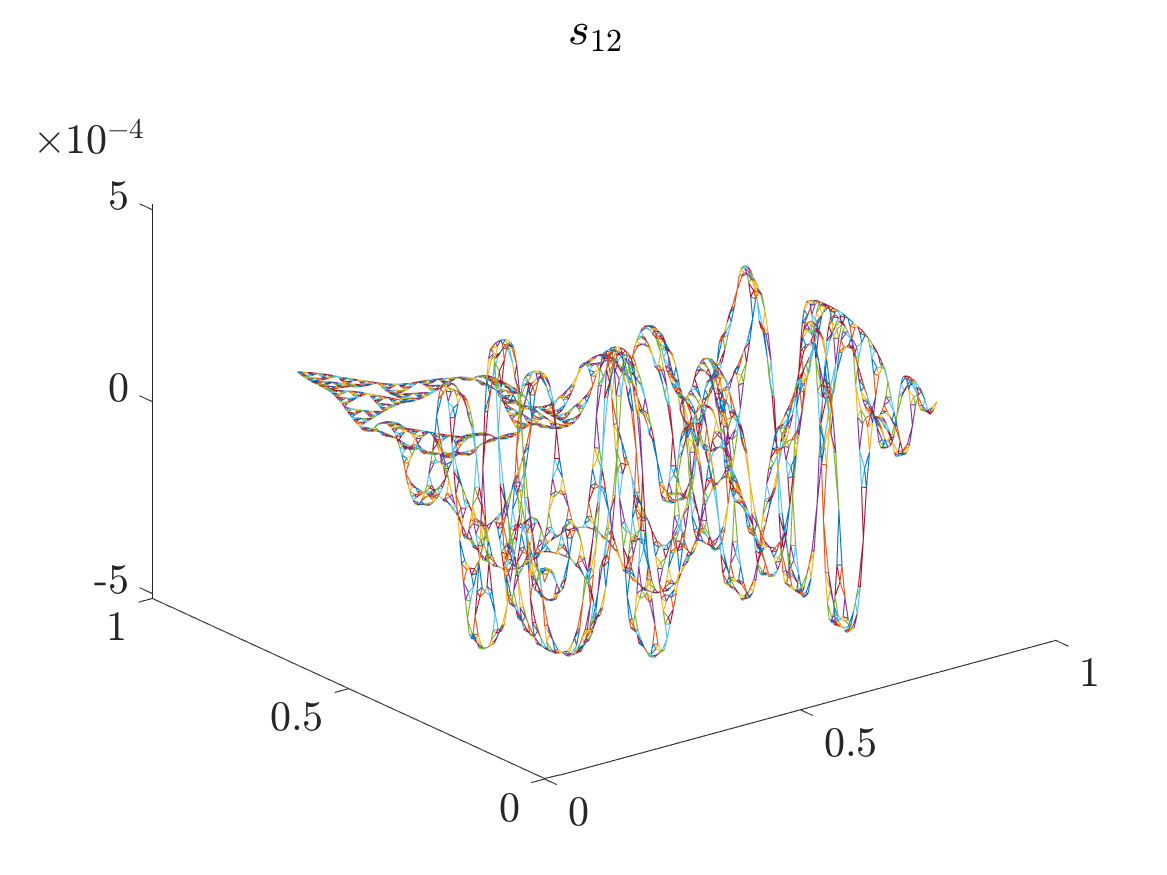}
    \includegraphics[width=0.45\textwidth]{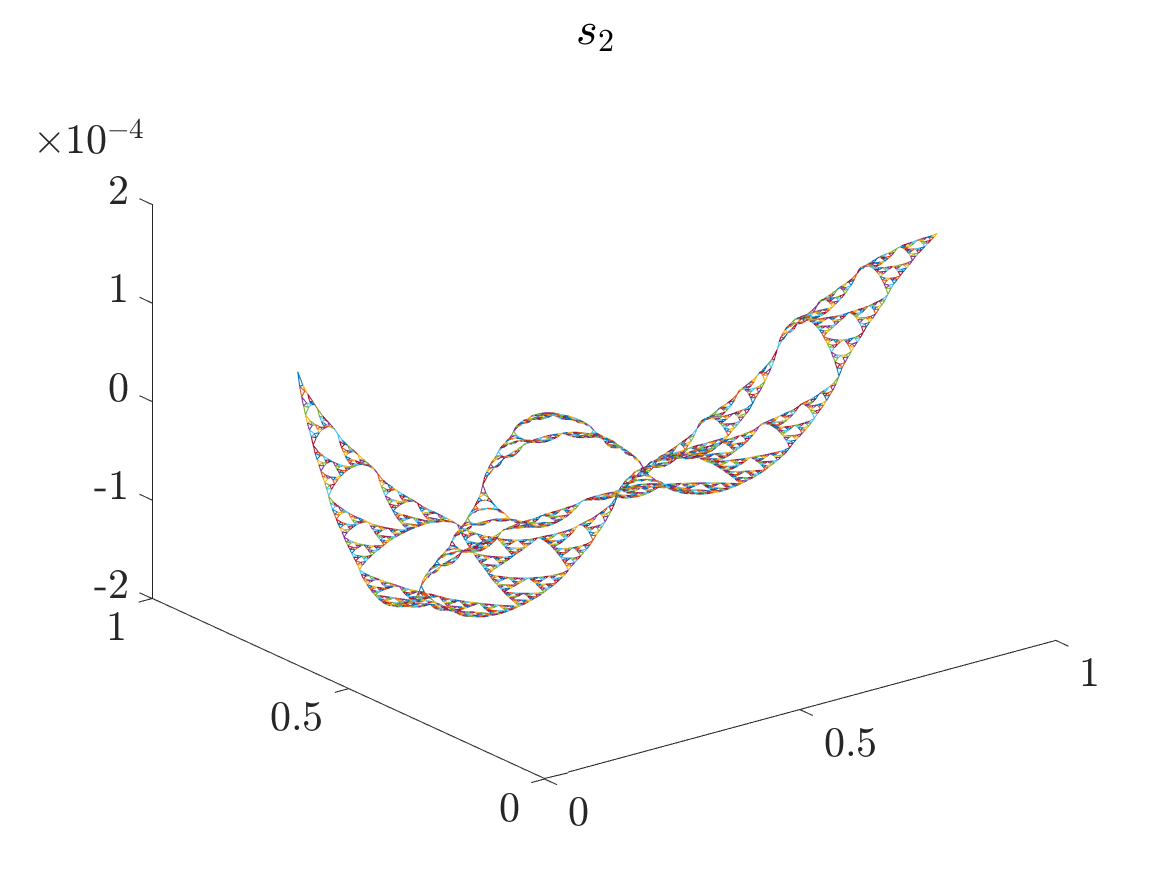}
    \includegraphics[width=0.45\textwidth]{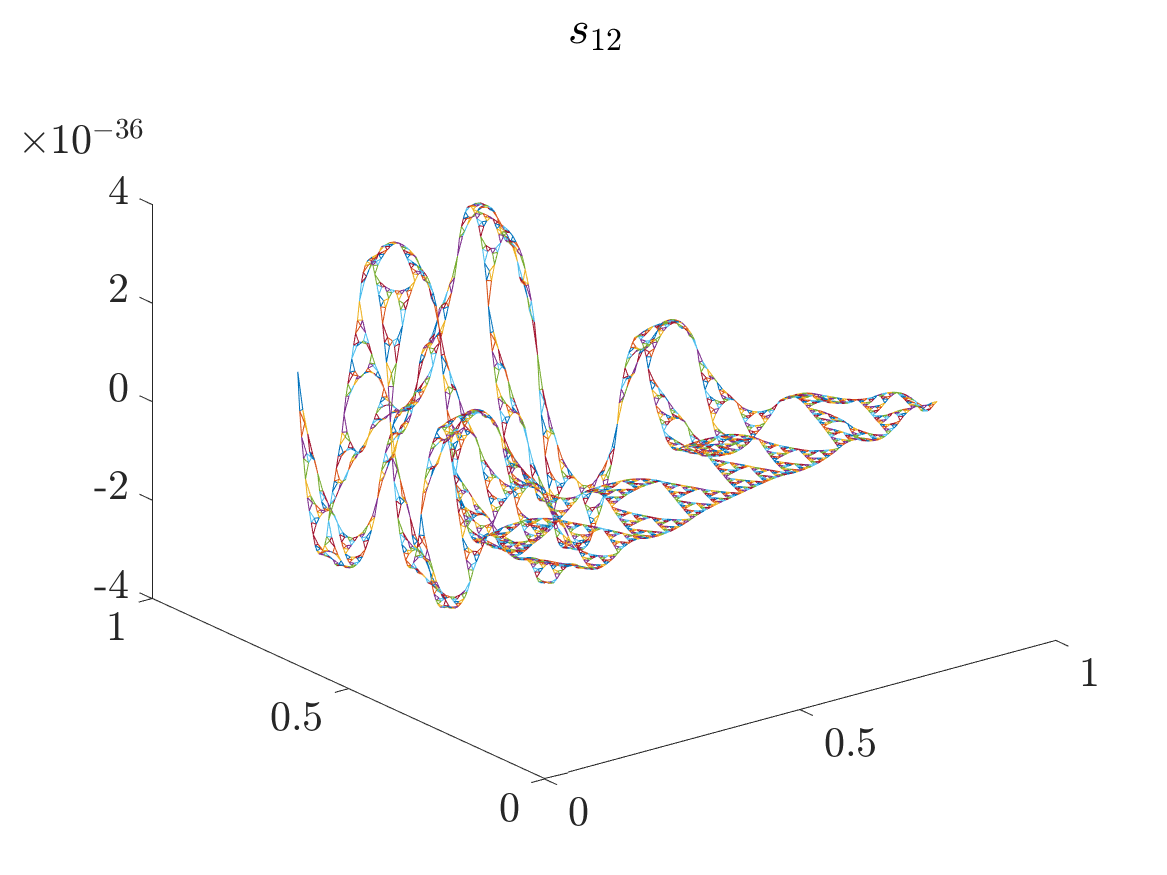}
   \caption{Plotting the Sobolev Orthogonal Polynomials --- top row: $k=3$, bottom row: $k=2$. For both cases we have $\chi = 1$.}
\label{fig: k=2,3 H1 plots}
\end{figure}

One initial observation is that the Sobolev polynomials are $4$ orders of magnitude smaller than the Legendre polynomials found in Figure 4 of \cite{OST}. This is due to the $L^{\infty}$ estimate given in Corollary \ref{Corollary: laplacian-L2 and Linfinity estimates}. The estimate also shows that $s_n$ decays to the zero polynomial uniformly as $n \rightarrow \infty$ due to the decay in $\|p_n\|_{L^2}$. Consequently in Figure \ref{fig: k=2,3 H1 plots} as the degree increases, the orders of magnitudes of the polynomials fall rapidly. In fact, for a sufficiently large $n$, the values taken by $s_n$ are arbitrarily small. To evaluate such high-degree polynomials without accumulating error we switch to computing in rational arithmetic. Our evaluation algorithm is as follows: we first evaluate the ``easy" basis $\qty{f_{j,k}}$ given in \cite[Equation 2.3]{SU} through \cite[Lemma 2.6]{SU}. Then we use \cite[Theorem 2.3]{NSTY} to convert from the easy basis to the monomial basis $\qty{P_{j,k}}$. Finally, after computing the coefficients $\qty{z_{l, n}}$ from Definition \ref{def:basicsop} using either the Gram-Schmidt process or \eqref{3-term recurrence}, we obtain the Sobolev polynomials $s_{n}$ by taking a linear combination of the evaluated monomials with the coefficients $\qty{z_{l, n}}$. The main drawback in this approach is that we can only evaluate $s_n$ on a graph approximation $V_m$ and with increasing $m$ the complexity of the recursion in \cite[Lemma 2.6]{SU} grows exponentially. Additionally, for small $n$, the coefficients $\qty{z_{n, l}}$ may be computed using a Gram-Schmidt routine. However, for large $n$, it is more advisable to use the recurrence relation instead. The complete code listing and documentation can be found at \cite{JSV}. 

We also initiate an investigation of the zero sets of the Sobolev-Legendre polynomials on $SG$ by taking the approach in \cite{OST}. It is known that  the Sobolev and Legendre polynomials have interlacing zeroes on $[-1,1]$, see \cite{MX} for details.
 Analogously, Figure~\ref{fig: Interlacing} is suggestive of ssimilar  mutual interlacing patterns between the zeroes of $p_n$ (Legendre OPs on $SG$), $s_n$ (Sobolev-Legendre OPs on $SG$). Notice that these interlacing properties are highly irregular on the edges of $SG$. Furthermore, the bottom row in Figure~\ref{fig: Interlacing} suggests that the zeros of $s_n$ may not all be simple: $s_{1,7}$ seems to have a zero of multiplicity greater than one. But we have not been able to prove or disprove this observation, even though we can use data at higher resolutions to lend credence to this guess.
\begin{figure}[h]
    \centering
  \includegraphics[width=.45\linewidth]{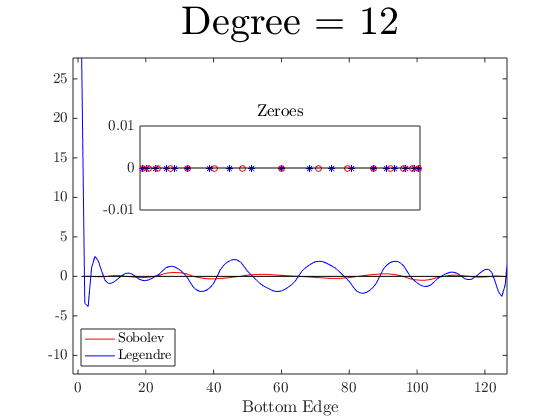}
  \includegraphics[width=.45\linewidth]{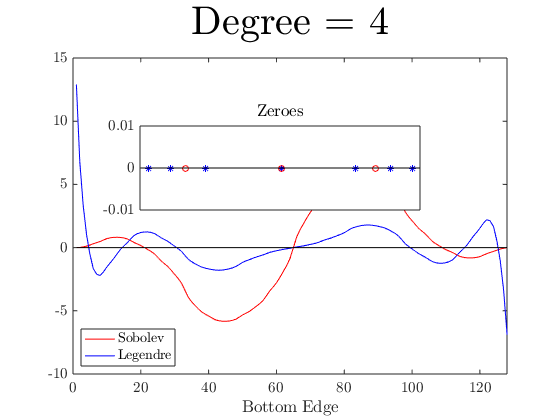}
  \includegraphics[width=.45\linewidth]{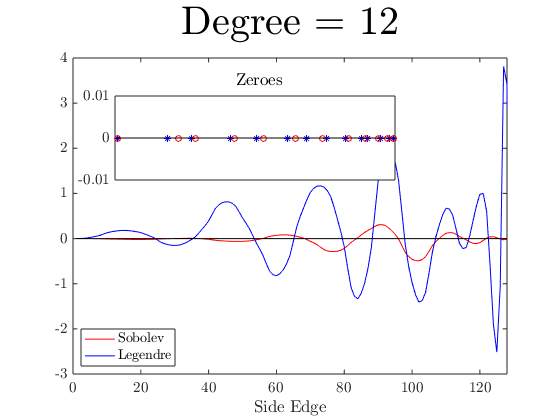}
  \includegraphics[width=.45\linewidth]{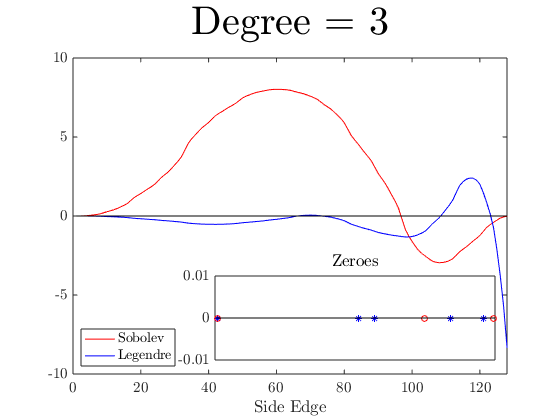}

  \caption{Interlacing patterns of $s_n$ and $p_n$ on the edges of SG --- top row: edge between $q_0$ and $q_1$, bottom row: edge between $q_1$ and $q_2$.}
  \label{fig: Interlacing}
\end{figure}

By the bottom edge, we mean the edge between $q_1$ and $q_2$ included. By a side edge, we mean the edge between $q_0$ and $q_i$ for $i=1,2$ including $q_0$ but not $q_i$ Our methodology in counting zeroes was rudimentary. We plotted the polynomials on $\Gamma_7$, which meant we had $129$ evaluation points on each edge. Then we simply computed the number of times the polynomial changed sign and concluded that by continuity the polynomial must have had a zero in the interval. There are two clear issues with this methodology: first, the polynomial may have more than one zero between two points of opposite sign. Secondly, this methodology cannot be used to compute non-simple zeroes. From the plots we observe that the edges look tangential to the polynomials at some points, implying the existence of high-multiplicity zeroes (HMZs). But we can only evaluate the polynomial at finitely many points. Consequently, sometimes, an HMZ can get trapped between evaluation points, so in our data it looks like the polynomial takes a non-zero value at the HMZ.

Using the above methodology of counting zeros, we plot the number of edge zeroes taken by Legendre and Sobolev polynomials against the degrees.  
\begin{figure}[h]
    \centering
    \includegraphics[width=0.7\textwidth]{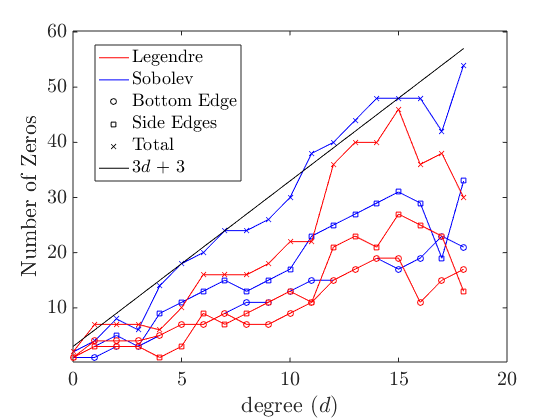}
    \caption{Number of zeroes of $p_d$ and $s_d$ on the edges of SG. Note that some high degree Sobolev polynomials have more zeroes than $\text{dim}\qty(\mathcal{H}_d) = 3d+3$ }
    \label{fig:my_label}
\end{figure}
We next plot Figure~\ref{fig: Antisymmetric plots} the symmetric Sobolev-Legendre OPs $\qty{s_n}$ obtained from the family of monomial corresponding to $k=1$. The construction of these OPs depends on Conjecture~\ref{conjecture:normal} which we have not been able to prove, but which numerical simulations suggest should be true. 

\begin{figure}[h]
\centering
    \begin{minipage}{.45\textwidth}
    \centering
    \includegraphics[width=\textwidth]{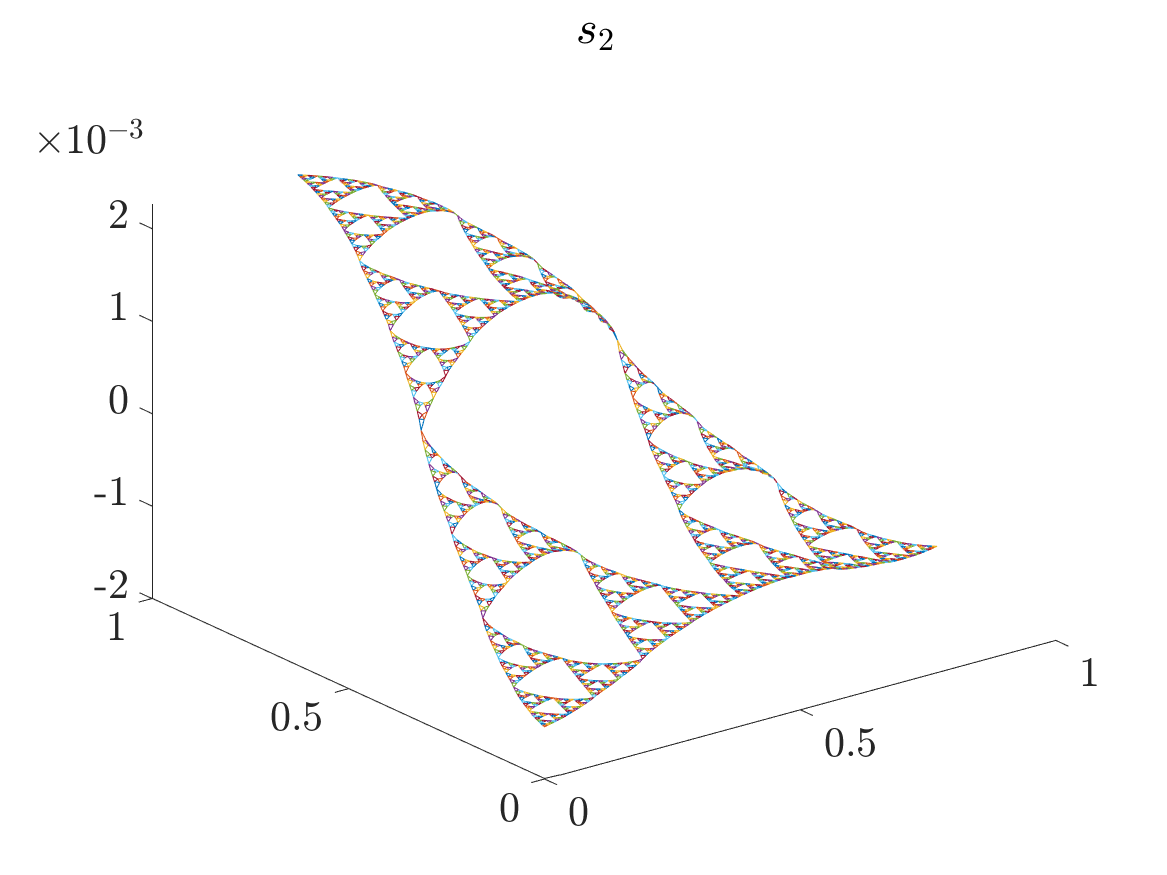}
    \end{minipage}
    \begin{minipage}{.45\textwidth}
    \centering
    \includegraphics[width=\textwidth]{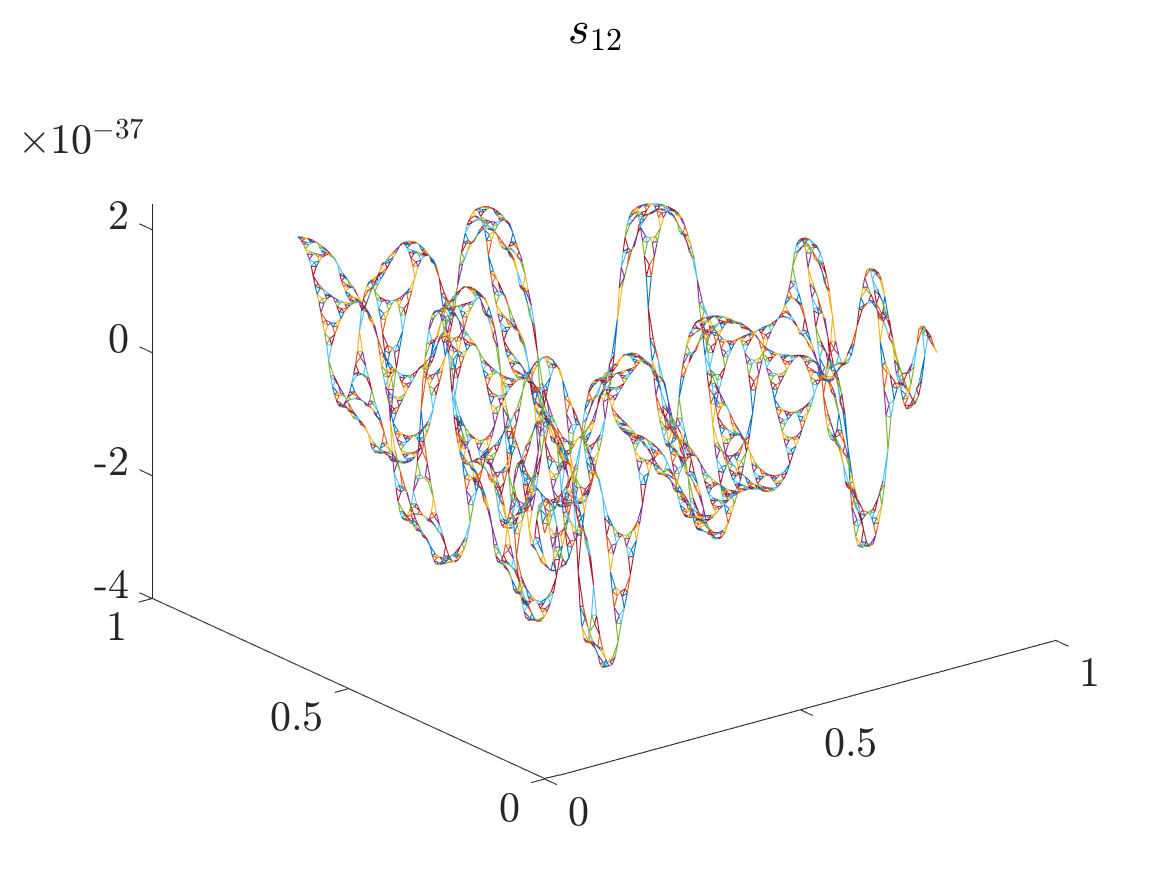}
    \end{minipage}
    \caption{Plots of the symmetric ($H^1$) Sobolev-Legendre orthogonal polynomials $s_j$ for $j=2,12$. Here we have $\chi = 1$ and $k=1$.}
            \label{fig: Antisymmetric plots}
\end{figure}

We also investigate the behavior of the SOPs when  $\chi \to \infty$ for the Sobolev inner product given by \eqref{eq: basic sobolev inner product} for $k=3$.
\begin{figure}[h]
    \centering
    \includegraphics[width=0.45\textwidth]{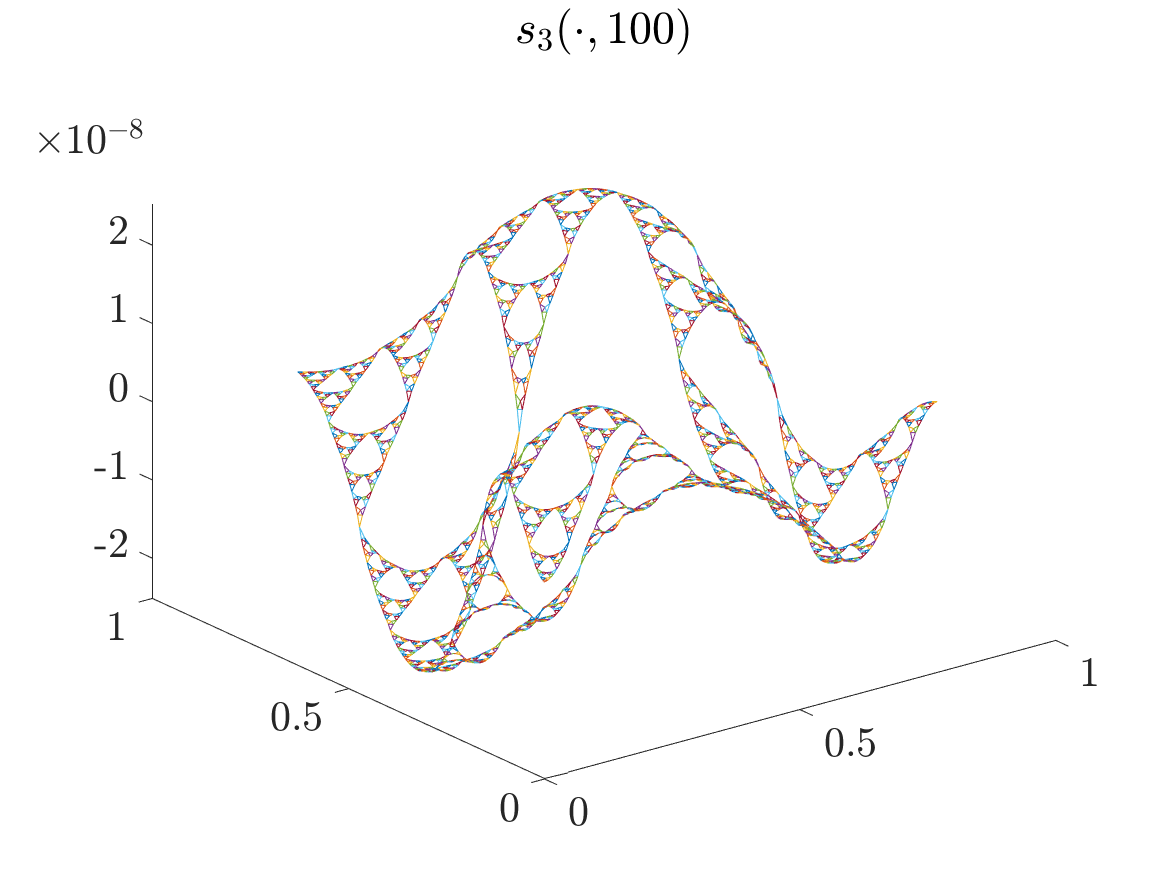}
    \includegraphics[width=0.45\textwidth]{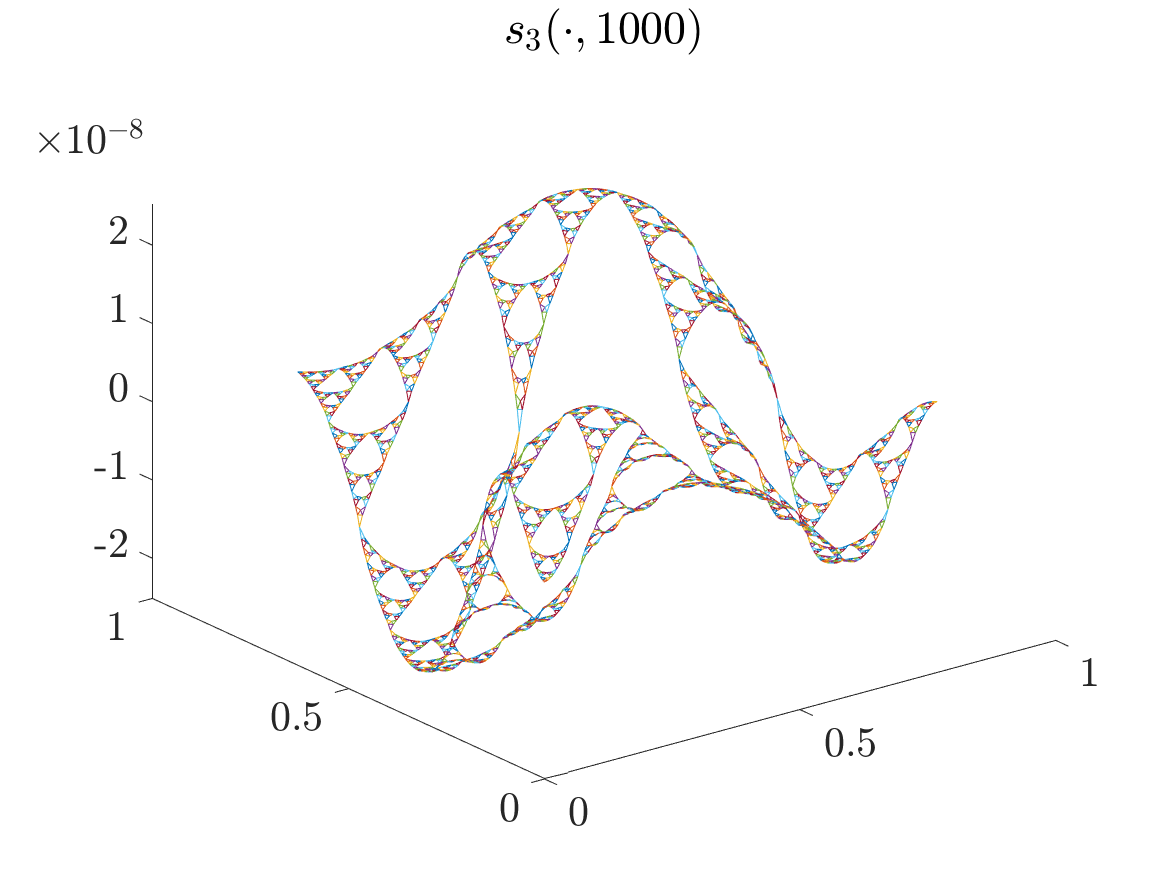}
    \includegraphics[width=0.45\textwidth]{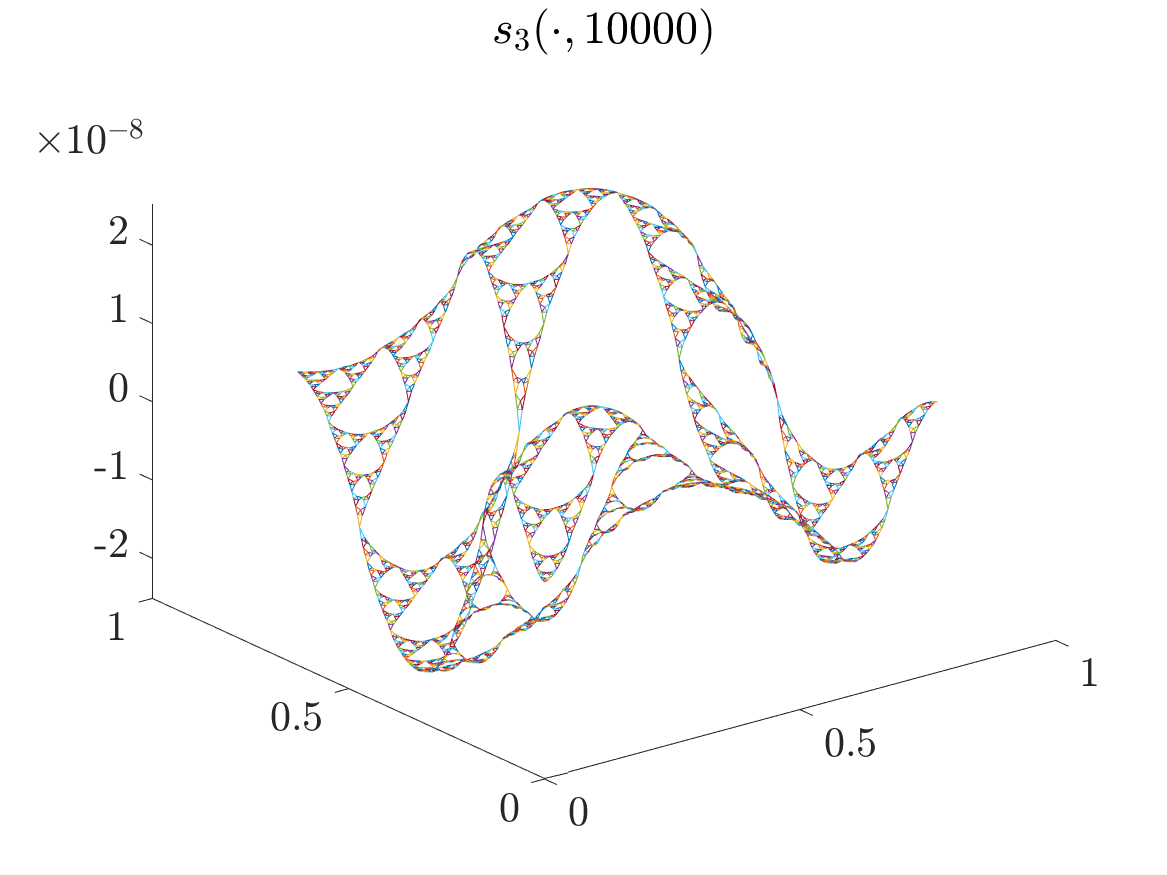}
    \includegraphics[width=0.45\textwidth]{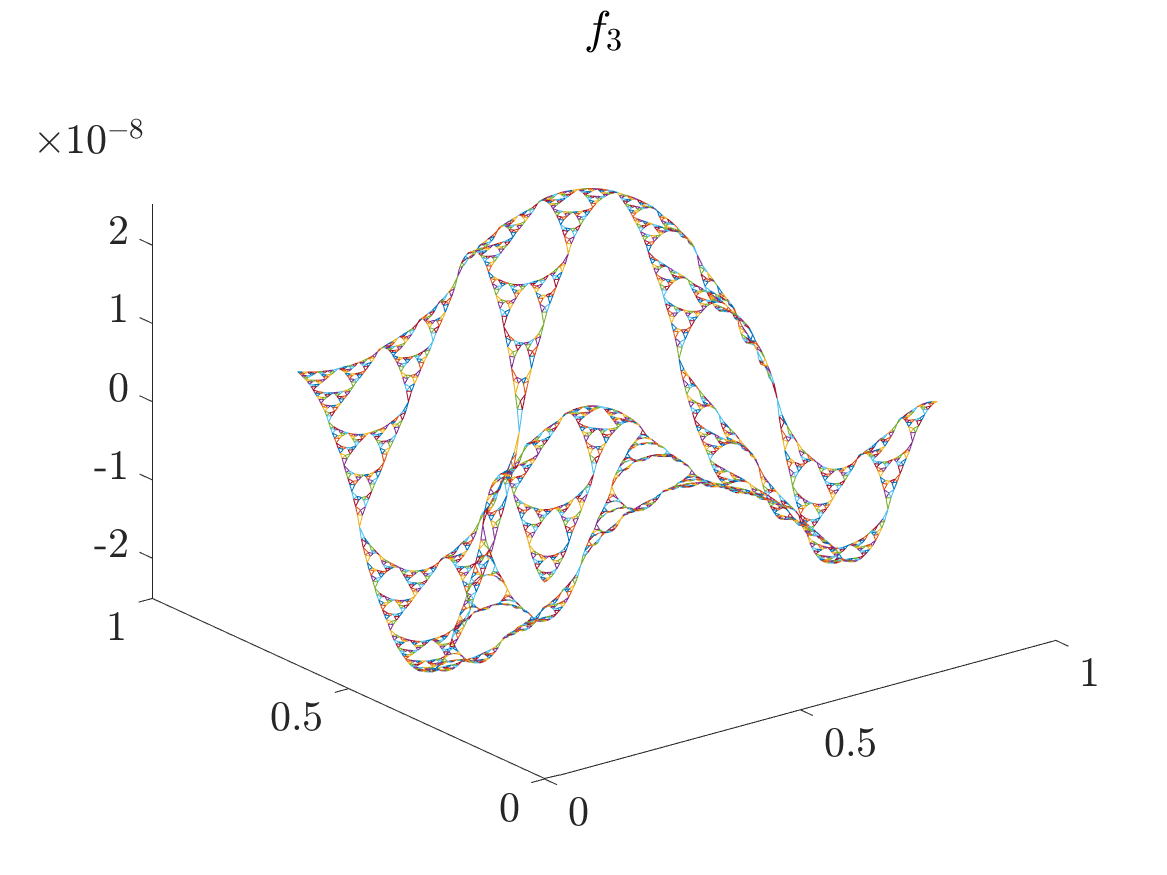}
    \caption{Studying $s_{3}\qty(\cdot,\chi)$ for $\chi = 10^m$ where $2\leq m \leq 4$. Note that here $k=3$. As $\chi \to \infty$ we can observe the convergent behaviour of the polynomial $s_{3}\qty(\cdot,\chi)$ to $f_3$ as outlined in Corollary \ref{Corollary: Convergence result for lambda}.}
    \label{fig:effect of lambda}
\end{figure}

Finally, we plot higher order Sobolev orthogonal polynomials for the $S^m$ inner product in Equation \eqref{eq:sobk}.

\begin{figure}[h]
    \centering
    \includegraphics[width=0.45\textwidth]{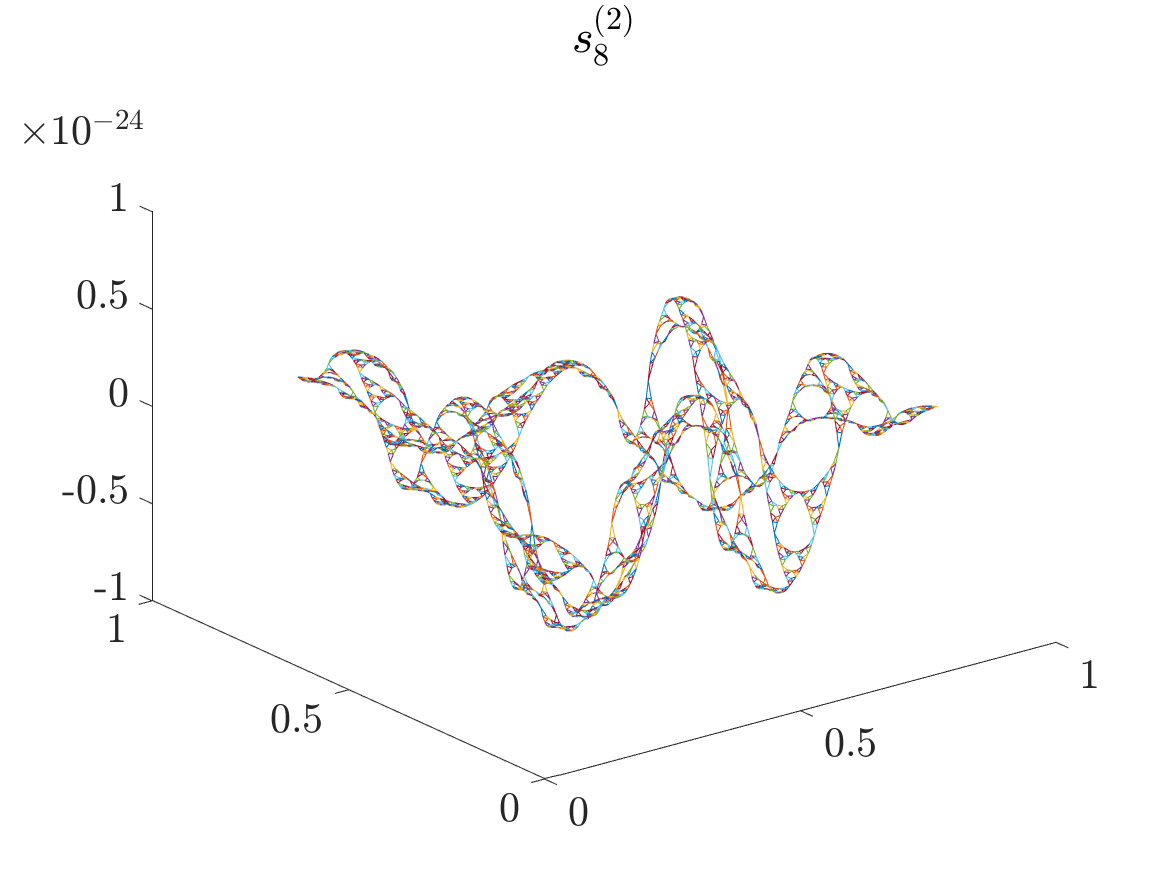}
    \includegraphics[width=0.45\textwidth]{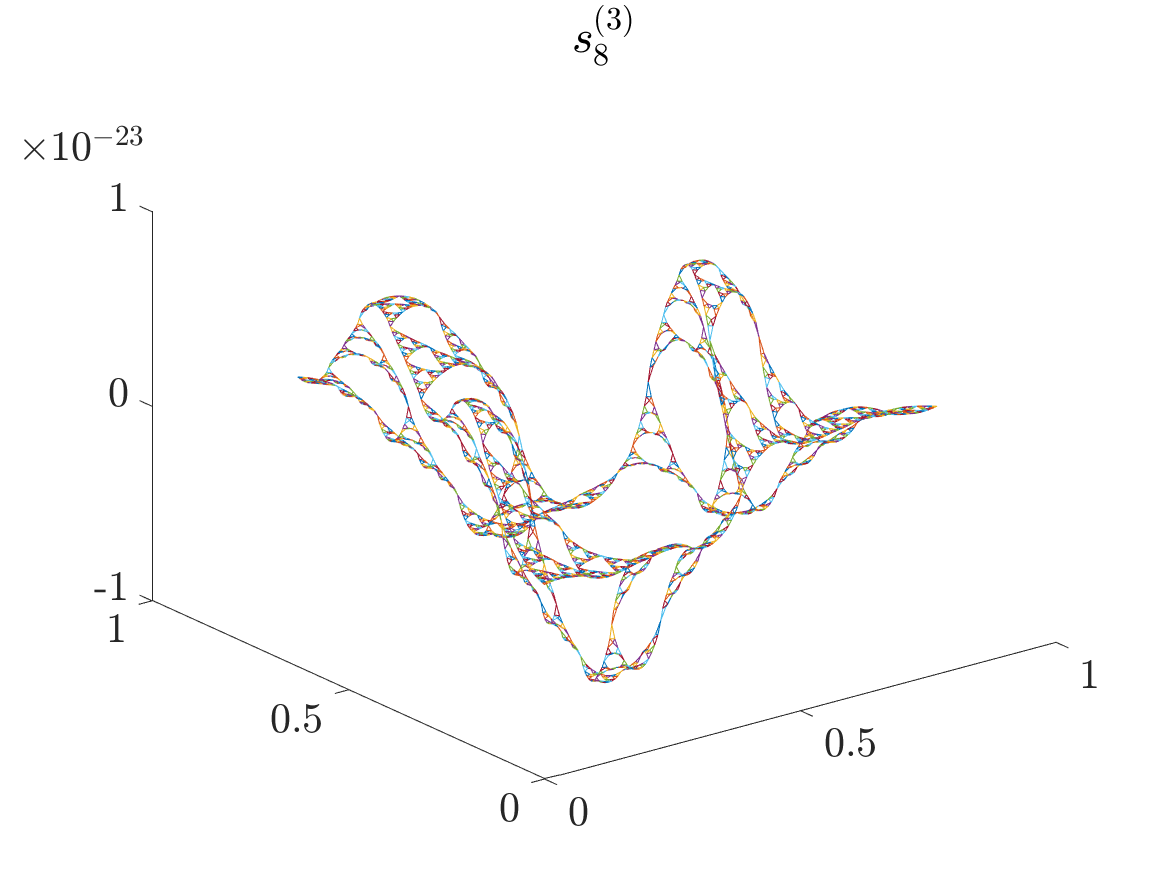}
    \includegraphics[width=0.45\textwidth]{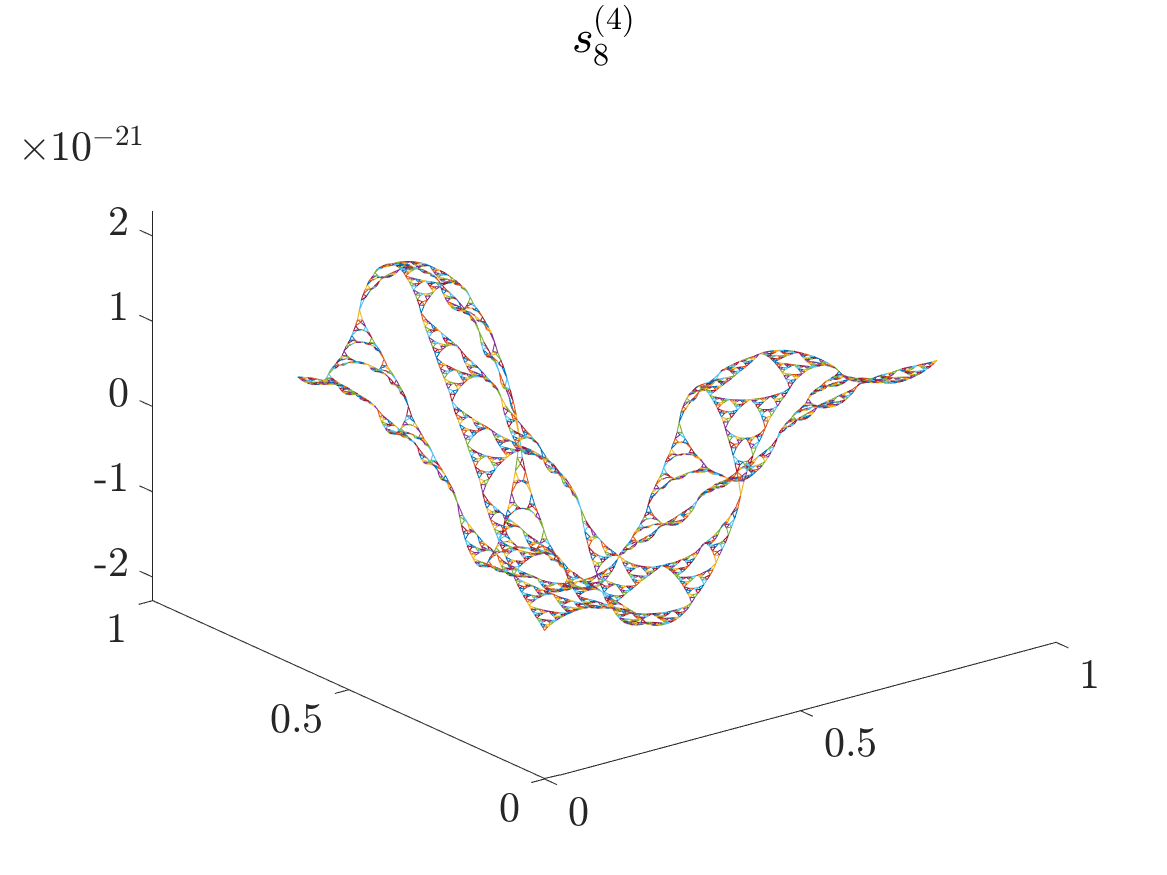}
    \includegraphics[width=0.45\textwidth]{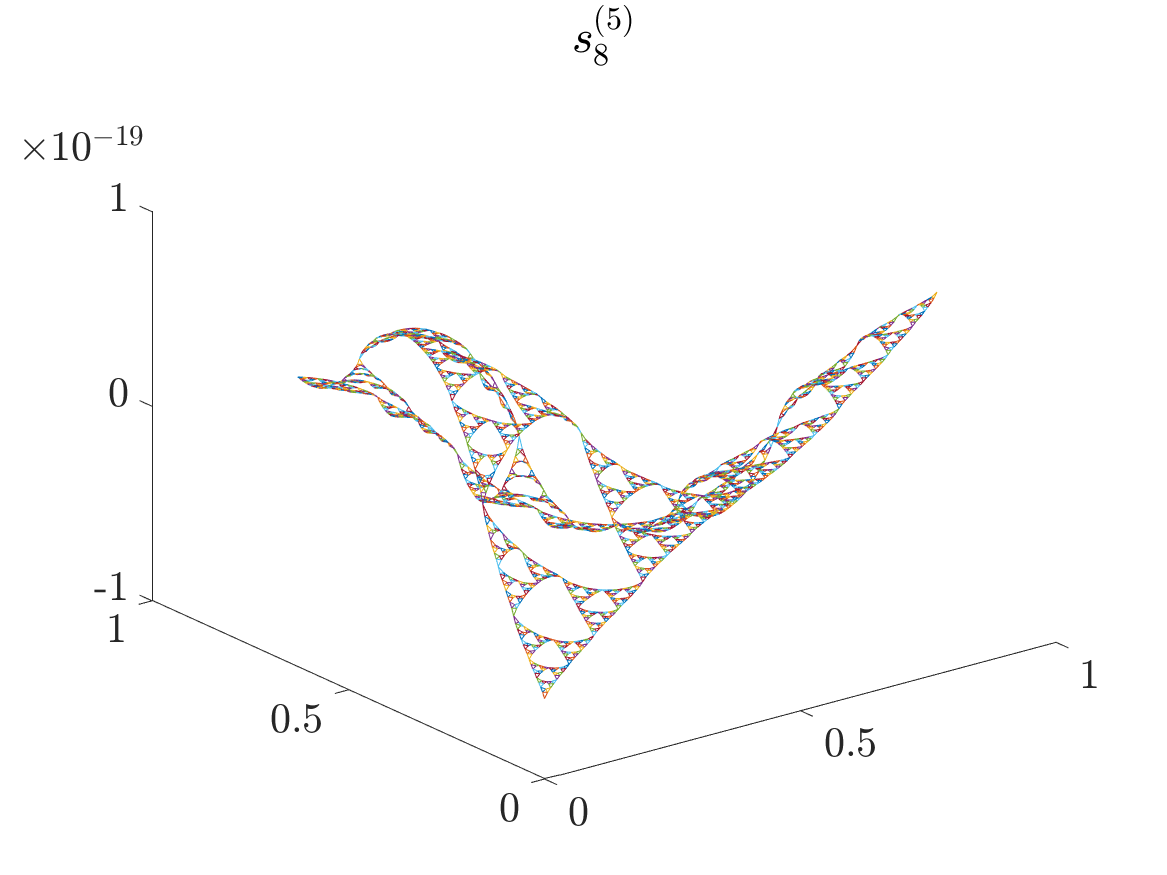}
    \caption{Visualizing $s_{8}^{\qty(m)}\qty(\cdot,X)$ where $m$ is the number of terms in the inner product in Equation \eqref{eq:sobk}. The above polynomial is the 8th degree polynomial obtained by applying the Gram-Schmidt process to $\qty{P_{j,3}}_{j\in \mathbb{N}}$ with the inner product in Equation \eqref{eq:sobk} for $2 \leq m \leq 5$. Note that we set $X = \textbf{1}_{m}$ and the energy and boundary terms in the inner product to be zero.}
    \label{fig:effect of m}
\end{figure}

We make two comments with regards to Figure \ref{fig:effect of m}. First, recall that fixing $\chi$ and a family of monomials $k$, the Sobolev orthogonal polynomials $\{s_{n,k}\qty(\cdot,\chi)\}$ converge to 0 as the degree $n$ increases. Furthermore, in \ref{fig:effect of lambda} the precision of $s_n\qty(\cdot, \chi)$ depletes with increase in $\chi$. This is not true for the parameter $m$: as $m$ increases, the eighth degree polynomial \textit{increases} in precision as we add more terms to the inner product. Second, consecutive orthogonal polynomials $s_{8}^{m}$ and $s_{8}^{m+1}$ are quite different in shape from each other. This may also be due to the use of small values of $m$. 

\subsection{Implementation and Code Design}
To generate plots of the orthogonal polynomials and test their properties numerically, we first compute the orthogonal polynomials directly from the monomial basis using the Gram-Schmidt process. In order to do this, we require the inner products between the monomials. Using the results of in \cite[Lemma 2.1]{OST}, we compute these inner products in terms of the coefficients $\alpha_j, \beta_j, \eta_j$, and $\alpha_j'$. To calculate the coefficients, we use the recursion relation from of \cite[Theorem B]{OST}. Moreover, all calculations are done in exact rational arithmetic and numerical values are only converted to floating point representations at the end of all calculations. Most of the computations involve recursions so the code has been memorised to improve efficiency. 

The values of $\alpha_j, \beta_j, \eta_j$, and $\alpha_j'$ are then used to calculate inner products between the monomial basis, and the results are stored in a Gram Matrix. Arbitrary polynomials of a certain family (value of $k$) are stored as their coordinate vectors in the monomial basis. Thus, inner products between arbitrary polynomials are computed as Euclidean inner products weighted by the Gram Matrix. This allows the Gram Schmidt algorithm to be implemented more efficiently.

For Sobolev-type inner products of any order, we first decompose inner products of Laplacians of monomials into inner products of monomials of lower degree and then build the Gram Matrix. For inner products involving the energy, the Gauss-Green formula is used to rewrite the inner product in terms of lower degree monomials.

The recursion relations for $k=2$ and $k=3$ involve computing the coordinates of the $f_{j,k}$ with respect to the monomial basis. To compute these, we use the results of Lemma~\ref{Lemma: Expansions for zeta}. This involves first computing the Legendre polynomials. Thus, we first compute the Legendre and Sobolev polynomials up to degree 2 using Gram-Schmidt. Then, we use the recursion relation in  \cite[Theorem 3.5]{OST} to calculate the remaining Legendre polynomials. Finally, we use the Legendre coordinates to compute $f_{j,k}$ and use those in the formula from \eqref{th:k23recurrence} to calculate the remaining Sobolev polynomials.

In the case that $k=1$, notice that the formula given in Lemma~\ref{Lemma: Expansions for zeta} for $f_{j,1}$ is very similar to the formula for $f_{j,2}$. The only differences are that $f_{j,1}$ has a nontrivial projection onto $P_{0,2}$ and $\zeta_{j, 1}$ depends on the $\alpha$ coefficients. However, we note that the prefactor of $\pm 2$ on the $\zeta_{j, 2}$ and $\zeta_{j, 3}$ are actually $-1/\beta_0$ and $-1/\gamma_0$ respectively. If we follow this pattern and create a function
$$ \Tilde{f}_{j+1,1} = \Tilde{\zeta}_{j, 1}P_{0,1} + \sum\limits_{l=0}^j\omega_{j,l} P_{l+1, k}\quad 
    \textrm{where} \quad \Tilde{\zeta}_{j, 1} = -\sum\limits_{l=0}^j\omega_{j,l}\alpha_{l+1},$$
we can actually use the recursion relation \eqref{th:k23recurrence} with $\tilde{f}$ in place of $f$ and $k=1$ to generate the Sobolev polynomials. 

Once we have the coordinate vectors of the orthogonal polynomials, we evaluate the monomials using the relations given in \cite{NSTY} and then generate the relevant plot. The plots in this paper were created in MATLAB for stylistic purposes.

\section{Applications}\label{sec: applications}
In this last section, we  explore two applications motivated by the study of zeroes of the Sobolev orthogonal polynomials. In Section~\ref{sec4.1} we consider the problem  of polynomial interpolation on $\sg$, while Section~\ref{sec4.2} treats the topic of quadrature rules for numerical integration on $\sg$.

\subsection{Polynomial Interpolation}\label{sec4.1}

We recall that for any set $\{(x_i, y_i)\}_{i=1}^{n+1}\subset \RR^2$, with $x_i\neq x_j$ for $i\neq j$, there exists a unique real polynomial $p$ of degree $n$ such that $p(x_i) = y_i$. Thus any $n$ degree polynomial belongs to a $d=n+1$ dimensional subspace of $P(\RR)$,  and it is uniquely determined by its values on $d$ distinct points. Motivated by this fact, we pose the following question on $\sg$: Is a degree $j$ polynomial uniquely defined by its values on finite set $E\subset \sg$? If so what is the cardinality of $E$ as compared to the degree $j$ of the polynomial or the dimension of the subspace in which it resides? To investigate this problem, it is enough to understand the zero set of a polynomial of degree $j$ on $\sg$. We further simplify the question by first studying the zeroes of degree $j$ anti-symmetric polynomials, which form a $j+1$-dimensional subspace of $\mathcal{H}_j$. 
For example, it appears that $s_{1,5}$ has $19$ zeroes in total on the bottom edge, and $27$ total side edge zeroes. Hence, $s_{1,5}$ seems to have at least $22$ zeroes in each half of $\sg$.
Let $x_0, \ldots x_{15}$ represent the first 16 of these zeroes. Then $s_{1,5}\qty(x_k)=0$ for $k=0, 1, \hdots, 15$ even though $s_{1,5}$ is not identically $0$. We can rewrite this to say the following matrix

\begin{align}
    \begin{bmatrix}
    P_{0,3}\qty(x_0) & \ldots & P_{15,3}\qty(x_0) \\
    \vdots & \ddots & \vdots \\
    P_{0,3}\qty(x_{15}) & \ldots & P_{15,3}\qty(x_{15})\\
    \end{bmatrix}
\end{align}

is singular. This statement is in stark contrast with polynomials on $\RR$, and seems to imply that the Fundamental Theorem of Algebra does not hold for polynomials on $SG$. 

Now we move to the general case. A general polynomial $f$ of degree $n$ is given by $$ f\qty(x) = \sum\limits_{j=0}^{n}\sum\limits_{k=1}^{3}c_{j, k}P_{j,k}\qty(x).$$ Consequently $f$ has $3n+3$ degrees of freedom (i.e it lies in a $3n+3$ dimensional subspace of $P(SG)$. We then ask if there exist sets of $3n+3$ distinct points on $\sg$ that uniquely determine every polynomial of degree $n$. This is equivalent to the existence of   sets of  $3n+3$ distinct points $E=\{x_1, \ldots, x_{3n+3}\}$ for which the following matrix is invertible

\begin{align}\label{interpolationmatrix}
    M_n = \begin{bmatrix}
    P_{0,1}\qty(x_1) & \ldots & P_{n,3}\qty(x_1) \\
    \vdots & \ddots & \vdots \\
    P_{0,1}\qty(x_{3n+3}) & \ldots & P_{n,3}\qty(x_{3n+3})\\
    \end{bmatrix}
\end{align}

We shall refer to the matrix  $M_n$  as the interpolation matrix on the set $\qty{x_1, \ldots, x_{3n+3}}$. It is easy to check that with the choice  $x_i=F_0^{\qty(i-1)}\qty(q_1)$ with $1\le i \le3n+3$,  then $\qty[P_{1,1}\qty(x_1)\ldots P_{1,1}\qty(x_{3n+3})]$ and   $\qty[P_{0,3}\qty(x_1)\ldots P_{0,3}\qty(x_{3n+3})]$ are colinear by scaling properties, hence the corresponding interpolation matrix is not invertible. 
We start by observing that when we take $n=1$ the points in $V_1$ completely determine any polynomial of degree $1.$  

\begin{lemma}\label{lem:interpolationv1} 
Let $g \in \mathcal{H}_1$. Then $g$ is determined uniquely by its values on $V_1$.
\end{lemma}

\begin{proof}
The proof is by direct computation of the interpolation matrix. We switch to the easy basis $\qty{f_{j,k}}$ where $j=0,1$ and $k=0,1,2$. Suppose $g|_{V_0} \equiv 0$. Then, $g = c_0f_{1,0} + c_1f_{1,1} + c_2f_{1,2}$. Now suppose $g|_{V_1} \equiv 0$. Then to check whether $c_i = 0$ we need to check the invertibility of

$$ \begin{bmatrix}
f_{1,0}\qty(F_0q_1) & f_{1,1}\qty(F_0q_1) & f_{1,2}\qty(F_0q_1) \\
f_{1,0}\qty(F_1q_2) & f_{1,1}\qty(F_1q_2) & f_{1,2}\qty(F_1q_2) \\
f_{1,0}\qty(F_2q_0) & f_{1,1}\qty(F_2q_0) & f_{1,2}\qty(F_2q_0)\\
\end{bmatrix}$$

But this is a circulant matrix and $f_{1,0}\qty(F_0q_1) + f_{1,1}\qty(F_0q_1) + f_{1,2}\qty(F_0q_1) = -1/15 \neq 0$ so it is invertible. 
\end{proof}

Unfortunately, the proof given for Lemma~\ref{lem:interpolationv1} does not generalize to higher order polynomials. However, under an assumption we have not been able to establish, the following set of $3n+3$ points uniquely determined any polynomial of degree $n$ on $\sg$.

\begin{lemma}\label{lemma:solinterpolation}
Suppose that each term in the  sequence  $\beta_j$ defined in Lemma~\ref{lem:recabg} never vanishes. 
For any $n\ge 0$, take $x_i=F_0^{\qty(i-1)}\qty(q_1)$ for $1\le i\le 2n+2$, and $x_i=F_0^{\qty(i-2n-3)}\qty(q_2)$ for $2n+3\le i \le 3n+3$. Then the matrix \eqref{interpolationmatrix} is invertible.
\end{lemma}

\begin{proof}
Suppose not, then there exists a  non-zero vector $\qty[a_1,\ldots,a_{3n+3}]$ such that $f\qty(x_i)=0$ where $f:= \sum\limits^{i=n+1}_{i=1}a_iP_{i-1,1}+a_{n+1+i}P_{i-1,2}+a_{2n+2+i}P_{i-1,3} = f_1 + f_2 + f_3$ where $f_k = \sum\limits^{i=n+1}_{i=1}a_{(k-1)(n+1)+i}P_{i-1,k}$. Note $f\qty(F_0^{\qty(i-1)}\qty(q_1))=f\qty(F_0^{\qty(i-1)}\qty(q_2))=0$ for $1\le i \le n+1$ so symmetry, we have $f_3\qty(F_0^{\qty(i-1)}\qty(q_1))=0$ for $1\le i \le n+1$. But notice the determinant of
$$M_n = \begin{bmatrix}
    P_{0,3}\qty(x_1) & \ldots & P_{n,3}\qty(x_1) \\
    \vdots & \ddots & \vdots \\
    P_{0,3}\qty(x_{n+1}) & \ldots & P_{n,3}\qty(x_{n+1})
\end{bmatrix} = \begin{bmatrix}
    \gamma_0 & \ldots & \gamma_{n} \\
    \vdots & \ddots & \vdots \\
    5^{-n}\gamma_0 & \ldots & 5^{-n^2}\gamma_n
\end{bmatrix}$$
is the product of some $\gamma_i$ (which are all positive) and the determinant of a Vandermonde matrix, which is $\prod\limits_{1\le i<j\le n+1}\qty(5^{-j}-5^{-i})$ by \cite[Equation 2.6]{NSTY}. It follows that $f_3=0$, and so $f$ has no anti-symmetric part. 

Using a similar argument along with \cite[Equations 2.4-2.6]{NSTY},  we can establish that all the coefficients in  $f=f_1 + f_2$ must vanish. This comes down to proving that the determinant of certain Vandermonde matrices are non zero. It is here that we need the fact that $\beta_j\neq0$. We note that values of $\beta_j$ for $j=0, \hdots, 20$ were given in~\cite[Table 1]{NSTY}. In addition, by~\cite[Theorem 2.9]{NSTY}, $\lim\limits_{j \to \infty}\qty(-\lambda_2)^j\beta_j= c\neq 0$, where $\lambda_2=135.572126995788...$ It follows that for $j$ large, we could assume $\beta_j\neq 0.$
\end{proof}

Finally, we note that in general the sets of points that could be used to uniquely define a polynomial is rather ``thin'' in the following sense. We first introduce the following notion that appeared in the investigation of  similar interpolation problems in the context of finite dimensional subspaces of $C\qty(\Omega)$ where $\Omega \subset \RR^n$ in \cite{M}. 

\begin{definition} Fix $n\geq 2$. 
A subset  $I \subseteq SG$  is called   $n$-interpolatory set of $SG$ if for any subset $N \subseteq I$ such that $|N| = 3n + 3$, $M_n$ is invertible on $N$.
\end{definition}

The next result, following a trick of Haar \cite{Meij}, shows that $n$-interpolatory sets on $SG$ have empty interior.

\begin{proposition}\label{prop: nint sets are thin}
For $n\geq 2,$ let $I_n \subseteq SG$ be an $n$-interpolatory set. Then $I_n$ cannot contain the three edges of a cell with $3n + 1$ additional points. In particular, it cannot contain a cell.
\end{proposition}

\begin{proof}
Suppose there is a cell $C$ such that all the three edges of $C$ lie in $I_n$. Let $I = \qty{a,b}$  where $a$ and $b$ are vertices in $C$. There exist two different paths $\gamma$ and $\eta$ joining $a$ and $b$. We may parametrize these paths as $\gamma\qty(t)$ and $\eta\qty(t)$ where $0 \leq t \leq 1$ such that $\gamma\qty(1) = \eta\qty(0) = a$ and $\gamma\qty(0) = \eta\qty(1) = b$.
Let $B = \qty{x_1, \ldots x_{3n + 1}}\subset I_n$, be any set of $3n + 1$ points not on $\gamma \cup \eta$. Thus, $S = B \cup I$ is a set of $3n + 3$ points and so $M_n$ is invertible on $S$ since $I_n$ is $n$-interpolatory.  Now for every $t$, $M_n$  stays invertible on $\qty{\eta\qty(t),\gamma\qty(t)} \cup B$  as we picked $B$ to not coincide with $\gamma$ and $\eta$. Traversing the two paths from $t=0$ to $t=1$ switches the rows of $M_n$ and hence the sign of its determinant. Thus, the determinant must vanish for some $T \in \qty(0,1)$. 
Consequently, $M_n$ is not invertible in the set $B \cup \qty{\gamma\qty(T), \eta\qty(T) } \subseteq I_n$, resulting in a contradiction.
\end{proof}

Interpolation of functions on graphs has been investigated in a variety of settings. We refer to \cite{Pesenson09, WNW20} and the references therein for more details.

\subsection{Quadrature on SG}\label{sec4.2}
In \cite{SU}, the authors prove a quadratic error bound for Simpson's rule on $\sg$, by interpolating a function using quadratic splines at level $m$. More generally, in analogy with Newton-Cotes rules on $\RR$, we may consider computing the integral of $f$ on $\sg$ by interpolating it on $V_m$ using splines of order $n$. But, as was discussed in the previous section, we may not be able to interpolate a function uniquely using splines of order $n$ using any selection of $3n + 3$ points on $V_m$ (Simpson's rule is a lucky case where $|V_1| = \text{dim}\qty(\mathcal{H}_1)$). However, the particular solution for the interpolation problem in Lemma~\ref{lemma:solinterpolation} allows us to prove the following estimate for a general spline quadrature rule: 

\begin{theorem}

\label{thm:nharmquad} Let $\{x_i\}_{i=1}^{3n+3}$ be defined as in Lemma~\ref{lemma:solinterpolation}. Given a quadrature rule $I_n^n\qty(f):=\sum\limits_{i=1}^{3n+3}\omega_if\qty(x_i)$ which exactly integrates functions in $\mathcal{H}_n$. Let $I_n^m\qty(f):=\sum\limits_{|\omega|=m-n}I^n_n\qty(f\circ F_\omega)$. Then we have the following estimate on the quadrature error:
\begin{align}
    \qty|I_n^m\qty(f) - \int_{SG}f\dd{\mu}|\leq c_1\qty(n)5^{-\qty(n+1)m}\|\Delta^{\qty(n+1)}f\|_\infty.
\end{align}
\end{theorem}

\begin{proof}
Break up $\int_{SG}f\dd{\mu}$ into integrals over cells $F_wSG$ where $|w| = m-n$. For one such cell, let $g_\omega\in \mathcal{H}_n$ be such that $g_\omega = f\eval_{V_n}$. Then we have
\begin{align}
    \qty|I_n^n\qty(f\circ F_w) - \int f\circ F_w\dd{\mu}| =  \qty|\int g_\omega - f\circ F_w\dd{\mu}|\\ \leq \|g_\omega - f\circ F_w\|_\infty \leq c_1\qty(n)\|\Delta^{\qty(n+1)}\qty(f\circ F_w)\|_\infty,
\end{align}
where the last equality results from applying $\qty(n+1)$-times Green operators to $\Delta^{\qty(n+1)}\qty(f\circ F_w)$ and making use of the interpolation and the properties of a finite-dimensional normed space. Combining the subintegrals over the cells and applying the triangle inequality results in
\begin{align}
     \qty|I_n^m\qty(f) - \int_{SG}f\dd{\mu}| = 3^{-\qty(m-n)} \qty|\sum\limits_{|w| = m-n} I_n^n\qty(f\circ F_w) - \int f\circ F_w\dd{\mu}|\\
     \leq c_1\qty(n)\sup\limits_\omega\|\Delta^{\qty(n+1)}\qty(f\circ F_w)\|_\infty \leq c_1\qty(n)5^{-\qty(n+1)m}\|\Delta^{\qty(n+1)}f\|_\infty.
\end{align}\end{proof}

\begin{remark}\label{rem:numerical-quadrature}
In practice, we find that using this construction for interpolation and quadrature is unstable due to Runge phenomena. The quadrature rules are exact for polynomials by construction, but using high order quadrature for other functions results in large numerical instabilities. To fix this problem, we could attempt to construct piecewise polynomial spline interpolants, which are much more stable. However, this construction would again involve determining a formula for extending a function $n$-harmonically.
\end{remark}


\section*{Acknowledgment}
 Kasso A. Okoudjou  was partially supported by the U. S.\ Army Research Office  grant W911NF1910366, and an MLK visiting professorship at MIT. Jiang, Lan, Sule, and Venkat would like to acknowledge the REU at Cornell University.   
 
\bibliographystyle{siam}
\bibliography{sob.bib}

\end{document}